\documentclass[11pt,a4paper,twoside]{article}
\usepackage{latexsym,amsfonts,amsmath, amsthm, amssymb}
\usepackage{fullpage}
\usepackage{xcolor,bm, bbm}
\usepackage[utf8x]{inputenc}
\usepackage{array}
\usepackage{mathrsfs}
\usepackage{fourier}
\usepackage{graphicx}

\newtheorem{theorem}{Theorem}
\newtheorem{lemma}[theorem]{Lemma}
\newtheorem{corollary}[theorem]{Corollary}

\newtheorem{conjecture}{Conjecture}

\theoremstyle{definition}
\newtheorem{remark}{Remark}

\newtheorem{thmalpha}{Theorem}

\newcommand{\N}{\mathbb{N}}
\newcommand{\R}{\mathbb{R}}

\newcommand{\IB}{\mathbb{B}}

\newcommand{\ID}{\mathbb{D}}
\newcommand{\IE}{\mathbb{E}}

\newcommand{\IN}{\mathbb{N}}
\newcommand{\IP}{\mathbb{P}}

\newcommand{\IR}{\mathbb{R}}

\newcommand{\lou}{\IB_{q,p}^n}
\newcommand{\lo}{\widetilde{\ID}_{q,p}^n}
\newcommand{\loo}{\widetilde{\ID}_{q,1}^n}
\newcommand{\loou}{\IB_{q,1}^n}
\newcommand{\lop}{\widetilde{\ID}_{q,p,+}^n}
\newcommand{\vol}{\mathrm{vol}}
\newcommand{\pr}{\mathscr{M}_1(\mathbb{R}_+)}
\newcommand{\dd}{\mathrm{d}}
\newcommand{\Cov}{\mathrm{Cov}}
\newcommand{\Var}{\mathrm{Var}}

\newcommand{\ed}{{\rm d}}

\allowdisplaybreaks
\begin{document}

\title{\bf A probabilistic approach to Lorentz balls}

\medskip

\author{Zakhar Kabluchko, Joscha Prochno and Mathias Sonnleitner}



\date{}

\maketitle

\begin{abstract}
\small
 We develop a probabilistic approach to study the volumetric and geometric properties of unit balls $\IB_{q,1}^n$ of finite-dimensional Lorentz sequences spaces $\ell_{q,1}^n$. More precisely, we show that the empirical distribution of a random vector $X^{(n)}$ uniformly distributed on the volume normalized Lorentz ball in $\R^n$ converges weakly to a compactly supported symmetric probability distribution with explicitly given density; as a consequence we obtain a weak Poincar\'e-Maxwell-Borel principle for any fixed number $k\in\N$ of coordinates of $X^{(n)}$ as $n\to\infty$. Moreover, we prove a central limit theorem for the largest coordinate of $X^{(n)}$, demonstrating a quite different behavior than in the case of the $\ell_q^n$ balls, where a Gumbel distribution appears in the limit. Last but not least, we prove a Schechtman-Schmuckenschl\"ager type result for the asymptotic volume of intersections of volume normalized Lorentz and $\ell^n_p$ balls.
\medspace
\end{abstract}


\section{Introduction}

The last decades have revealed a deep connection between the geometry of high-dimensional convex bodies and probability theory, and each field fruitfully influenced the other. Probabilistic methods naturally come into play when one considers a (high-dimensional) convex body in $\IR^n$, i.e., a compact and convex set with non-empty interior, as a probability space when it is endowed with the canonical Borel $\sigma$-field and the normalized uniform measure.

Probably the most prominent family of convex bodies are the unit balls $\IB_p^n$ of the classical finite-dimensional sequence spaces $\ell_p^n$ ($1\leq p\leq \infty$). This parametric family of bodies is arguably one the most studied ones in geometric functional analysis and their analytic and geometric properties are quite well understood today. In numerous instances it is the already mentioned probabilistic point of view that gives access to understanding the asymptotic structure as the space dimension $n$ tends to infinity, because there is a rather simple probabilistic representation of the uniform distribution on $\IB_p^n$. This representation allows one to go from a random vector uniformly distributed on $\IB_p^n$, and thus from one having dependent coordinates for $p<\infty$, to a random vector whose entries are independent and identically distributed according to the so-called $p$-Gaussian distribution. This heavily facilitates computations and is therefore one of the key tools used in the study of $\ell_p^n$ balls. It were Schechtman and Zinn \cite{SZ1990}, and independently Rachev and R\"uschendorf \cite{RR1991}, who showed that if $X=(X_1,\dots,X_n)$ is distributed uniformly at random on $\IB_p^n$, then
\begin{align}\label{eq:schechtman-zinn representation}
X \stackrel{\ed}{=} U^{1/n}\frac{Y}{\|Y\|_p},
\end{align}
where $U$ is distributed uniformly on $[0,1]$ and $Y=(Y_1,\dots,Y_n)$ is independent of $U$ with $Y_1,\dots,Y_n$ being independent and identically distributed with Lebesgue density on $\IR$ given by
\[
f_p(x):=
\begin{cases}
\frac{1}{2p^{1/p}\Gamma(1+\frac{1}{p})}e^{-|x|^p/p} &: 1 \leq p < \infty \cr
\frac{1}{2} \mathbbm 1_{[-1,1]}(x) &: p=\infty.
\end{cases}
\]
Here and below, $\stackrel{\ed}{=}$ denotes equality in distribution. The previous result was lifted in \cite{BGMN2005} to a wider class of distributions related to $\ell_p^n$ balls that include the uniform distribution and the distribution with respect to the cone probability measure as special cases. As we mentioned above, the representation presented as well as its generalization are frequently used in the asymptotic analysis of geometric and volumetric aspects of those unit balls and we refer, for instance, to \cite{APT2019,APT2021,GKR2017,JP2022,KPT2019_I,N2007,NR2003,SS1991,SZ2000} and the survey article \cite{PTT2020}.

A natural and frequently studied generalization of $\ell_p$ spaces (and their function space counterparts), which is classical not only in functional analysis \cite{BGVV2014,LT1977}, harmonic analysis \cite{HS1975} and optimization \cite{KMW2011}, is the class of Orlicz spaces. The finite-dimensional Orlicz space $\ell_M^n$ is $\IR^n$ endowed with the norm
\[
\|(x_i)_{i=1}^n\|_M := \inf\Bigg\{ \rho>0\,:\, \sum_{i=1}^n M\Big(\frac{|x_i|}{\rho} \Big) \leq 1 \Bigg\},
\]
where $M:\IR \to\IR$ is symmetric, convex, and satisfies both $M(0)=0$ and $M(x)>0$ for $x\neq 0$ (when $M(t)=|t|^p$, we have $\|\cdot\|_M=\|\cdot\|_p$). As in the case of $\ell_p^n$ spaces, the Orlicz spaces belong to the important class of finite-dimensional $1$-symmetric Banach spaces, i.e., spaces $(X,\|\cdot\|_X)$ with a basis $e_1,\dots,e_n\in X$ such that
\[
\Big\|\sum_{i=1}^n x_ie_i\Big\|_X = \Big\|\sum_{i=1}^n \varepsilon_ix_{\pi(i)}e_i \Big\|_X
\]
holds for all $x_1,\dots,x_n\in\IR$, all signs $\varepsilon_1,\dots,\varepsilon_n\in\{-1,1\}$, and all permutations $\pi$ of the numbers $\{1,\dots,n\}$; Orlicz spaces are intensively studied in the functional analysis literature and we refer to \cite{HKTJ2006,K1984,KS1985,PS2012,Sch1995} and references cited therein. Trying to lift or generalize results that can be obtained for the spaces $\ell_p^n$ by employing the Schechtman-Zinn probabilistic representation \eqref{eq:schechtman-zinn representation} one soon hits a dead end, because such representation in the case of unit balls $\IB_M^n$ of Orlicz spaces $\ell_M^n$ is not known. As a consequence, generalizations of results regarding the asymptotic structure of $\IB_p^n$ (as mentioned above) remained inaccessible for quite a while. Recently, Kabluchko and Prochno \cite{KP2021}, using maximum entropy considerations in the framework of non-interacting particles that have their origin in statistical mechanics, derived an \emph{asymptotic} version of a Schechtman-Zinn type representation for Orlicz spaces, relating the uniform distribution on Orlicz balls to Gibbs distributions with potential given by the Orlicz functions. This connection allowed to obtain a number of further results in the general setting of Orlicz balls, wich can be found, for instance, in \cite{AP2022,BW2021,FP2021,JP2023,KLR2022}.

Alongside Orlicz spaces, the second natural class of generalizations of $\ell_p$ spaces are Lorentz spaces, which were introduced by George Lorentz in the 1950s \cite{L1950,L1951}. In fact, it had already been observed by Marcinkiewicz in \cite{M1939} that Lebesgue spaces are not sufficient to capture the fine properties of operators on $L_p$ spaces. Since their introduction, Lorentz spaces have found numerous applications in different areas of mathematics such as approximation theory \cite{dVL1993}, harmonic analysis \cite{G2014}, interpolation theory \cite{BS1988}, and signal processing \cite{FR2013}. We are interested in the finite-dimensional Lorentz space $\ell_{q,p}^n$ with $1\leq p\leq q \leq \infty$, which is merely $\IR^n$ with the norm
\[
\|(x_i)_{i=1}^n\|_{q,p}:= \Big(\sum_{i=1}^{n}|i^{1/q-1/p}x_i^*|^p\Big)^{1/p},
\]
where $x_1^*,\dots,x_n^*$ is the non-increasing rearrangement of the numbers $|x_1|,\dots,|x_n|$ (when $q=p$, we have $\|\cdot\|_{q,p}=\|\cdot\|_p$). Those spaces again belong to the important class of $1$-symmetric Banach spaces and have also been intensively studied in the functional analysis literature, for instance, in \cite{A1975,DV2020,HM2006,JMST1979,P2020,R1981,R1982,Sch1984,Sch1989,TJ1989}. Indeed, unexpected phenomena can occur in these spaces; for example, as proved in \cite{EN11}, they are a counterexample to a conjecture on the interpolation behaviour of entropy numbers which holds true for $\ell_p$ spaces. Results regarding a probabilistic approach to the asymptotic volumetric and geometric structure of unit balls $\IB_{q,p}^n$ in Lorentz spaces $\ell_{q,p}^n$, of the same flavor as the ones mentioned for $\IB_p^n$ or $\IB_M^n$ above, are completely absent from the literature. One reason is arguably that the non-increasing rearrangement that appears in the definition of the norm makes any analysis quite delicate.

The motivation of our paper is thus twofold, namely to develop a first probabilistic approach on the one hand and apply this to study some asymptotic properties of Lorentz spaces $\ell_{q,p}^n$ on the other. In view of the vast literature regarding a probabilistic take on $\ell_p^n$ or more generally $\ell_M^n$ spaces and their unit balls, this is a first step towards approaching numerous natural questions that arise, like central limit theorems, conditional limit theorems, or moderate and large deviations \cite{APT2019,APT2021,FP2021,GKR2017,JP2022,KPT2019_I,KR18}, asymptotic volume ratios and thin-shell measure concentration \cite{AP2022,KP2021}, asymptotic independence of coordinates \cite{BW2021} or asymptotic thin-shell results \cite{KLR2022}, just to mention a few.

\section{Main results}\label{sec:main results}

We shall now present our main results, starting with the probabilistic approach to the uniform distribution on Lorentz balls and continuing with the applications to geometric and volumetric properties. The results are restricted to the case of Lorentz balls with parameters $p=1$ and $1<q\leq\infty$, but we present a conjecture together with heuristic arguments on what asymptotic probabilistic representation to expect in the general setting.

We shall denote by
\[
\lo
:=n^{1/q}\lou
=\Bigg\{x\in\IR^n\colon \frac{1}{n}\sum_{i=1}^{n}\Big(\frac{i}{n}\Big)^{p/q-1}|x_i^*|^p\le 1\Bigg\}
\]
for $q<\infty$ and
\[
\widetilde{\ID}_{\infty,1}^n
:=\log(n+1)\IB_{\infty,1}^n
=\Bigg\{x\in\IR^n\colon \frac{1}{\log(n+1)}\sum_{i=1}^{n}\frac{1}{i}|x_i^*|\le 1\Bigg\}.
\]
for $q=\infty$ the normalized unit balls such that, asymptotically up to constants, they have Lebesgue volume 1. Indeed, it follows from a classical result on the volume of unit balls in Banach spaces due to Sch\"utt \cite{Sch1982} that $\vol_n(\IB_{q,1}^n)^{1/n}\asymp_q n^{1/q}$ and, as shown in \cite[Theorem 7]{DV2020}, for $q=\infty$ we have $\vol_n(\IB_{\infty,1}^n)^{1/n} \asymp \log(n+1)^{-1}$; here $\asymp_q$ denotes equivalence up to constants depending only on $q$ and $\asymp$ equivalence up to absolute constants. 

The following stochastic representation constitutes the basis for our results and is of independent interest. It is the probabilistic entry point into the study of the asymptotic structure of Lorentz balls in the case $p=1$. 

\begin{thmalpha}\label{thm:stochastic-rep}
	Let $1\le q\le \infty$, $n\in\IN$ and assume $X^{(n)}$ is uniformly distributed on $\loou$. Then
\[
X^{(n)}\overset{\rm d}{=}\Bigg(\varepsilon_1\frac{\sum_{j=\pi(1)}^{n}\kappa_{q}(j)^{-1}E_j}{\sum_{j=1}^{n+1}E_j},\dots,\varepsilon_n\frac{\sum_{j=\pi(n)}^{n}\kappa_{q}(j)^{-1}E_j}{\sum_{j=1}^{n+1}E_j}\Bigg),
\]
where $E_1,\dots,E_{n+1}$ are standard exponential random variables, $\varepsilon=(\varepsilon_1,\dots,\varepsilon_n)$ is uniformly distributed on $\{-1,1\}^n$, $\pi$ is uniformly distributed on the symmetric group $S_n$ of permutations of $\{1,\dots,n\}$, and for $j\in\{1,\dots,n\}$,
\[
\kappa_{q}(j)
:=\sum_{i=1}^{j}i^{1/q-1}.
\]
All random objects are independent of each other.
\end{thmalpha}

According to Theorem \ref{thm:stochastic-rep}, the first (or any) coordinate of a random vector uniformly distributed on $\loo=n^{1/q}\loou$ for $q<\infty$ is equal in distribution to
\begin{equation} \label{eq:stochastic-rep}
\varepsilon_1\frac{\frac{1}{n}\sum_{i=u_n}^{n}\frac{n^{1/q}}{\kappa_q(i)}E_i}{\frac{1}{n}\sum_{i=1}^{n+1}E_i},
\end{equation}
where $u_n$ is uniformly distributed on $\{1,\dots,n\}$ and $\varepsilon_1$ is uniformly distributed on $\{-1,1\}$.

\begin{remark}\label{rem:one}
Note that the special case $q=1$ of Theorem~\ref{thm:stochastic-rep} corresponds to the $\ell_1^n$ ball and is consistent to the well-known fact that the order statistics of $n$ independent standard exponential random variables are distributed as
\[
\frac{E_n}{n},\frac{E_n}{n}+\frac{E_{n-1}}{n-1},\dots,\sum_{j=1}^{n}\frac{E_j}{j}
\]
beginning with the smallest ($n^{\text{th}}$) order statistic up to the largest ($1^{\text{st}}$) (see, e.g., \cite[Eq. (2.5.5)]{DN03}).
\end{remark}
As a consequence of the proof of Theorem \ref{thm:stochastic-rep}, we can deduce the following precise asymptotics for the volume radius; the following can also be obtained by refining the calculations from \cite{DV2020}. 

\begin{corollary}\label{cor:volume}
Let $1\le q \le\infty$. As $n\to\infty$, we have
\begin{equation} \label{eq:asymptotics}
\vol_n(\loou)^{1/n}
\sim 
\begin{cases}
	\frac{2}{q}e^{1/q}n^{-1/q}&: 1\le q<\infty,\\
	2(\log n)^{-1}&: q=\infty.
\end{cases}
\end{equation}
\end{corollary}

In our next theorem, we establish a weak convergence result for the empirical distribution. It is also used to establish the asymptotic distribution of a single coordinate presented afterwards.

\begin{thmalpha} \label{thm:main}
	Let $1<q\le\infty$ and for each $n\in\IN$ assume that $\widetilde{X}^{(n)}=(\widetilde{X}_1^{(n)},\dots,\widetilde{X}_n^{(n)})$ is a random vector uniformly distributed on $\loo$. Then, for every bounded continuous function $f\colon \IR\to\IR$, we have
\[
\frac{1}{n}\sum_{i=1}^{n}f(\widetilde{X}_i^{(n)})\xrightarrow[n\to\infty]{\IP} \int_{\IR} f(x)\, \nu_{q,1}(\dd x),
\]
where $\nu_{q,1}$ is a symmetric probability measure on $\IR$ with Lebesgue density $f_{q,1}\colon \IR\to\IR$ given by
\[
f_{q,1}(x) := \frac{1}{2}
\begin{cases}
q(1-(q-1)|x|)^{1/(q-1)}\mathbbm{1}_{[-\frac{1}{q-1},\frac{1}{q-1}]}(x) & \colon  q<\infty\\
\mathbbm{1}_{[-1,1]}(x) & \colon q=\infty.
\end{cases}
\]
\end{thmalpha}

Essentially due to exchangeability of the coordinates, it follows from Theorem~\ref{thm:main} that any fixed choice of $k\in\IN$ coordinates is asymptotically distributed as the $k$-fold product of $\nu_{q,1}$ as the dimension of the ambient space tends to infinity. We will obtain the following weak Poincar\'e-Maxwell-Borel principle for normalized Lorentz balls.

\begin{corollary}\label{cor:coordinates}
	Let $1<q\le\infty$ and for each $n\in\IN$ assume that $\widetilde{X}^{(n)}=(\widetilde{X}_1^{(n)},\dots,\widetilde{X}_n^{(n)})$ is uniformly distributed on $\loo$. For every $k\in\IN$ and with $\nu_{q,1}$ as in Theorem~\ref{thm:main}, for any bounded and continuous function $f\colon \IR^k\to\IR$, we have
\[
\IE \big[f(\widetilde{X}_1^{(n)},\dots,\widetilde{X}_k^{(n)}) \big] \stackrel{n\to\infty}{\longrightarrow} \int_{\IR^k}f(x)\, \nu_{q,1}^{\otimes k}(\dd x),
\]
that is, $(\widetilde{X}_1^{(n)},\dots,\widetilde{X}_k^{(n)})$ converges in distribution to a vector $Y^{(k)}=(Y_1,\dots,Y_k)\sim \nu_{q,1}^{\otimes k}$.
\end{corollary}
\begin{remark}
The case $q=1$ for which $\IB_1^n=\IB_{1,1}^n$ is already known \cite{RR1991} and in this case the asymptotic coordinate distributions are two-sided exponential. By considering $q_N:=1+1/N\to 1$, we see that it is consistent with our result due to limiting behavior
\[
f_{q_N}(x)=\frac{1+1/N}{2}(1-|x|/N)^{N}\mathbbm{1}_{[-N,N]}(x)\stackrel{N\to\infty}{\longrightarrow} \frac{1}{2}\exp(-|x|),\quad x\in\R.
\]
\end{remark}
\begin{remark}
	A notable difference in the asymptotic probabilistic behavior of vectors chosen uniformly at random from the volume normalized versions of $\IB_{q,1}^n$ and $\IB_q^n$ is that in the former case the asymptotic distribution of a single coordinate has compact support whereas in the latter case one obtains unbounded support (see \eqref{eq:schechtman-zinn representation} and the definition of the $q$-Gaussian distribution).
\end{remark}

In contrast to the maximum norm of uniform random vectors in normalized $\ell_q^n$ balls, which has Gumbel fluctuations \cite[Theorem 1.1 (c)]{KPT2019_I}, we have a central limit theorem for the Lorentz balls for $q>2$, i.e., while the volume radius in terms of the dimensions is the same in both cases, the probabilistic behavior is quite different. In between, that is, for $q\in (1,2]$, we have a different type of limit theorem which seems to interpolate between the Gumbel distribution and the normal distribution.

\begin{thmalpha}\label{thm:clt-inf}
Let $1\le q< \infty$, $r\in (0,\infty)$ and for each $n\in\IN$ assume that $\widetilde{X}^{(n)}$ is uniformly distributed on $\loo$. Then we have:
\begin{enumerate}
	\item[(i)] For $1\leq q< 2$
		\[
		n^{1-1/q}\big(\|\widetilde{X}^{(n)}\|_{\infty}-\mu_{q,n}\big)\xrightarrow[n\to\infty]{\ed} R_q
		\overset{\ed}{=}\sum_{j=1}^{\infty}\frac{E_j-1}{\kappa_q(j)},
		\]
	\item[(ii)] For $q=2$
		\[
		\frac{\sqrt{n}}{\sqrt{\log n}}\big(\|\widetilde{X}^{(n)}\|_{\infty}-\mu_{2,n}\big)\xrightarrow[n\to\infty]{\ed} R_2 \sim \mathscr{N}(0,1/4),
		\]
	\item[(iii)] For $2<q<\infty$
\[
\sqrt{n}\big(\|\widetilde{X}^{(n)}\|_{\infty}-\mu_{q,n}\big) \xrightarrow[n\to\infty]{\ed} \mathscr{N}(0,\sigma_q^2).
\]
\end{enumerate}
Here, $R_1+\gamma$ has a Gumbel law with distribution function $x\mapsto e^{-e^{-x}}$, $\gamma$ is the Euler-Mascheroni constant,
\[
\mu_{q,n}:=\frac{1}{n}\sum_{j=1}^{n}\frac{n^{1/q}}{\kappa_q(j)}\qquad \text{and}\qquad \sigma_q^2:=\frac{1}{q(q-1)^2 (q-2)}. 
\]
\end{thmalpha}


We can also use the probabilistic representation to study the $\ell_r^n$ norm of a random vector uniformly distributed in $\ID_{q,1}^n:=\vol_n(\mathbb{B}_{q,1}^n)^{-1/n}\mathbb{B}_{q,1}^n$, i.e., we determine its asymptotic length with respect to $\|\cdot\|_r$; we shall apply that result later to study the volumetric behavior of certain intersections. We have the following weak law of large numbers.

\begin{thmalpha}\label{pro:lln-norm}
Let $1<q\le \infty$, $1<r\le\infty$, and for each $n\in\IN$ assume that $X^{(n)}$ is uniformly distributed on $\ID_{q,1}^n$. Then
\[
n^{-1/r}\|X^{(n)}\|_r \xrightarrow[n\to\infty]{\mathbb P}  m_{q,r},
\]
where for $r<\infty$ we have
\[
m_{q,r}:=\frac{1}{2e^{1/q}}\frac{q}{q-1}\left(\frac{\Gamma(r+1)\Gamma\Big(1+\frac{q}{q-1}\Big)}{\Gamma\Big(r+1+\frac{q}{q-1}\Big)}\right)^{1/r} \text{ if } q<\infty \quad\text{and}\quad m_{\infty,r}:=\frac{1}{2} \Big(\frac{1}{r+1}\Big)^{1/r} \text{ if } q=\infty,
\]
and for $r=\infty$ we have
\[
m_{q,\infty}:=\frac{1}{2e^{1/q}}\frac{q}{q-1}\text{ if } q<\infty \quad\text{and}\quad m_{\infty,\infty}:= \frac{1}{2}.
\]
\end{thmalpha}

As a consequence we can deduce the following Schechtman-Schmuckenschl\"ager type result on the asymptotic volume of intersections of Lorentz and $\ell_r^n$ balls, thereby complementing the results from \cite{JKP2022,JP2022,KP2021,KPT2019_I,KPT2020_matrix,SS1991,SZ1990,S2001}. In the following corollary, we shall denote $\ID_r^n:=\vol_n(\mathbb{B}_r^n)^{-1/n}\mathbb{B}_r^n$.

\begin{corollary}\label{pro:intersect}
Let $1<q\le \infty$  and $1<r\le\infty$. Then, for all $t>0$,
\[
\vol_n(\ID_{q,1}^n\cap t\ID_r^n)\stackrel{n\to\infty}{\longrightarrow}
\begin{cases}
	1 & \colon A_{q,r}t>1\\
	0 & \colon A_{q,r}t<1,\\
\end{cases}
\]
where for $r<\infty$, we have
\[
A_{q,r}:=
\begin{cases}
	\frac{e^{1/q-1/r}\frac{q-1}{q}}{\Gamma(1+1/r)r^{1/r}}\left(\frac{\Gamma(r+1+\frac{q}{q-1})}{\Gamma(r+1)\Gamma(1+\frac{q}{q-1})}\right)^{1/r}& \colon q<\infty\\
\frac{1}{\Gamma(1+1/r)}\left(\frac{r+1}{re}\right)^{1/r}& \colon q=\infty,\\
\end{cases}
\]
and for $r=\infty$ we have
\[
A_{q,\infty}
:=
\begin{cases}
e^{1/q}\frac{q-1}{q}& \colon q<\infty\\
1& \colon q=\infty.\\
\end{cases}
\]
\end{corollary}

\begin{remark}
Comparing with analogous results for the $\ell_p^n$ balls obtained in \cite{SS1991}, we see that for $q=\infty$ and $r<\infty$ the obtained threshold is the same. Therefore, asymptotically $\ID_{\infty,1}^n$ behaves somewhat like $\ID_{\infty}^n$ when intersected with an $\ell_r^n$ ball.

Also for $r<\infty$, using $\Gamma(x+\alpha)\sim \Gamma(x)x^{\alpha}$ for $x\to\infty$, we obtain that
\[
\lim_{q\to 1}A_{q,r}=\frac{e^{1-1/r}}{\Gamma(1+1/r)\Gamma(r+1)^{1/r}r^{1/r}},
\]
which equals the threshold one would get for $\ID_{1}^n$ intersected with $\ID_r^n$. That is, in the boundary cases $q\in \{1,\infty\}$ we have that $\ID_{q,1}^n$ behaves similarly to $\ID_q^n$. In between, however, the behaviour of the threshold constants appears to be different, for example as $q\to\infty$ the constant $A_{q,r}$ grows for $\ID_q^n$ but converges for $\ID_{q,1}^n$.
\end{remark}


For parameters $p>1$, we currently do not have a probabilistic representation. However, we can use heuristic arguments based on maximum entropy considerations in order to derive the following conjecture about the limiting distribution.

\begin{conjecture}\label{conj:allp}
	Let $1\le p \le q < \infty$ and for each $n\in\IN$ assume that $\widetilde{X}^{(n)}=(\widetilde{X}^{(n)}_1,\dots,\widetilde{X}^{(n)}_n)$ is uniformly distributed on $\lo$. Then, for every bounded continuous function $f\colon \IR\to\IR$, we have
\[
\frac{1}{n}\sum_{i=1}^{n}f(\widetilde{X}_{i}^{(n)})\xrightarrow[n\to\infty]{\IP} \int_{\IR} f(x)\, \nu_{q,p}(\dd x),
\]
where $\nu_{q,p}$ is a symmetric probability measure on $\IR$ absolutely continuous with respect to Lebesgue measure and with density function $f_{q,p}\colon \IR\to [0,\infty)$ satisfying $f_{q,p}(x)=\frac{1}{2}G'(|x|)$, where $G\colon [0,\infty)\to\IR$ satisfies $G(0)=0$ and is the unique solution to the differential equation
\begin{align*} \label{eq:ode}
	G''(x)&=-G'(x)\big(1-G(x)\big)^{p/q-1}x^{p-1},\quad x\in (0,r_{p,q}),
\end{align*}
with $G(0)=0$ and $\lim_{x\uparrow r_{p,q}}G'(x)=0$, where $r_{p,q}\in (0,\infty]$ is the first point such that $G(r_{p,q})=1$. Further, we conjecture that $r_{p,q}=\infty$ if and only if $p=q$.
\end{conjecture}

Varying the free parameter $G'(0)>0$, we obtain different solutions, see Figure~\ref{fig:ode} for a simulation. If $G'(0)$ is smaller than some critical value $c_{p,q}$, then we conjecture that $\lim_{x\to\infty}G(x)<1$. If $G'(0)=c_{p,q}$, then there is a (minimal) value $r_{p,q}$ such that $G(r_{p,q})=1$ and we have $\lim_{x\uparrow r_{p,q}}G'(x)=0$. This critical solution is the one that appears in the above conjecture.


\begin{figure}[h]
	\begin{center}
		\includegraphics[width=.75\textwidth]{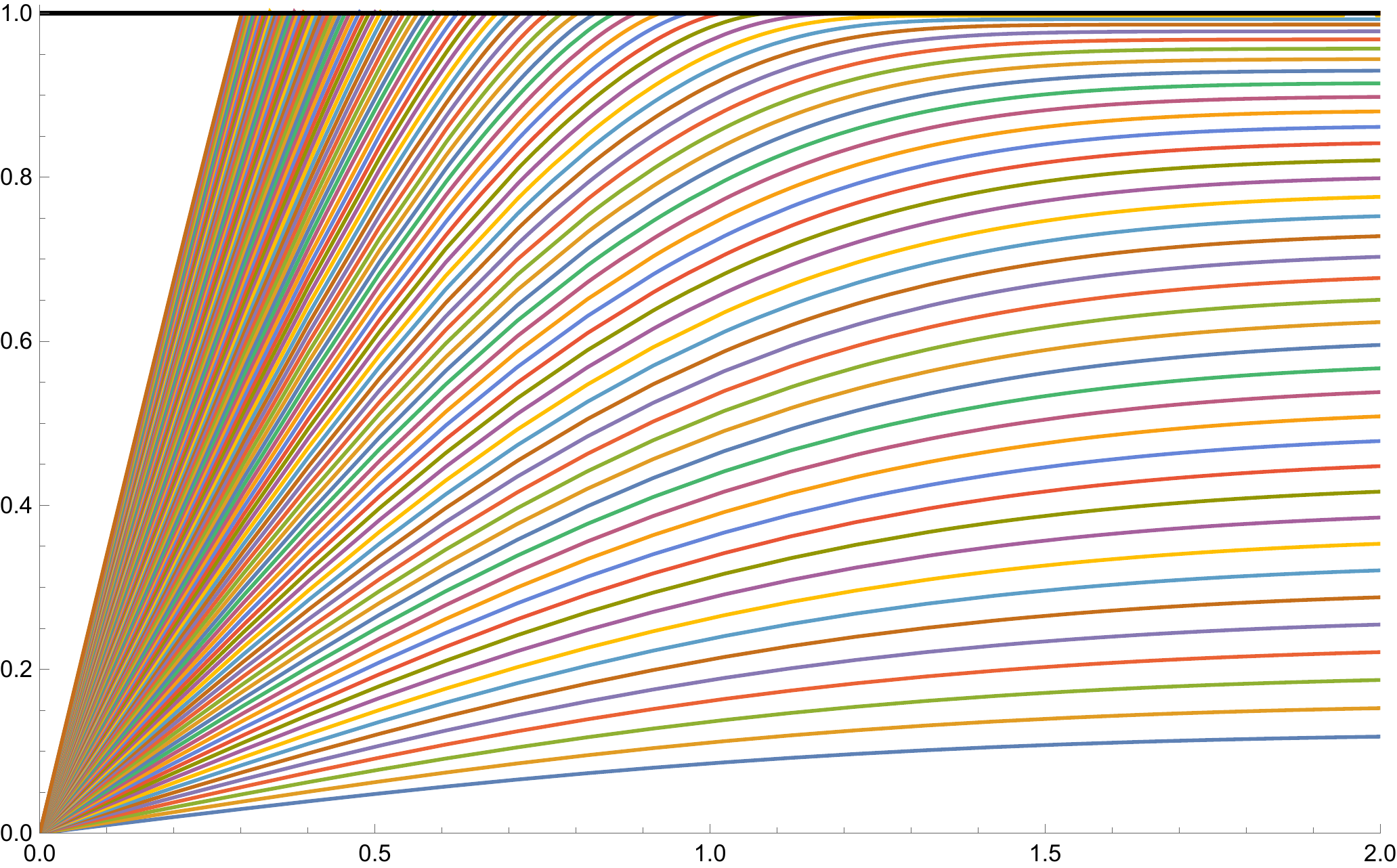}
	\end{center}
	\caption{Simulation of solutions of the differential equation in Conjecture~\ref{conj:allp} with different values of $G'(0)$.}
	\label{fig:ode}
\end{figure}

\begin{remark}
	The conjecture encapsulates the case $p=q$ for $\ell_p^n$ balls, where it correctly returns $p$-Gaussian densities, and the case $p=1$, where it gives the limiting distribution obtained in Theorem~\ref{thm:main}.
\end{remark}

\section{Proofs of the main results}

We shall now present the proofs of our main results presented in Section \ref{sec:main results}. In what follows, given a set $A\subseteq \IR^n$, we shall denote by ${\rm conv}(A)$ the convex hull of the set $A$ and by $\partial A$ its boundary. Moreover, we denote by $e_1,\dots,e_n$ the standard unit vectors in $\IR^n$.

\subsection{Proof of Theorem \ref{thm:stochastic-rep}}

Note that due to the 1-symmetry of Lorentz norms (and thus unit balls) it is sufficient to look at the Weyl chamber
\begin{equation}\label{eq:weyl chamber}
W:=\big\{x\in\IR^n\colon x_1\ge \cdots \ge x_n\ge 0\big\}
\end{equation}
intersected with $\loou$. This set then takes the form
\[
\loou \cap W
=\Big\{x\in W\colon \sum_{i=1}^{n}i^{1/q-1}x_i\le 1\Big\}.
\]

In a first step, we determine the extreme points (vertices) of the polytope $\loou\cap W$.

\begin{lemma} \label{lem:extreme}
Let $1<q\le\infty$. The extreme points of $\loou\cap W$ are given by
\[
0,Me_1,Me_2,\dots,Me_n,
\]
where
\[
M:=
\begin{pmatrix}
	\kappa_{q}(1)^{-1}& \kappa_{q}(2)^{-1}&\cdots &  \kappa_{q}(n)^{-1}\\
	0 &\kappa_{q}(2)^{-1}& \cdots &  \kappa_{q}(n)^{-1}\\
	0 & 0 & \ddots & \vdots\\
	0 &0& \cdots &  \kappa_{q}(n)^{-1}\\
\end{pmatrix}\quad \text{with}\quad
\kappa_{q}(j)
=\sum_{i=1}^{j}i^{1/q-1},\quad j\in\{1,\dots,n\}.
\]
\end{lemma}
\begin{proof}
First, we note that the Weyl chamber $W$ is the positive hull of the summing basis
\[
Ae_1=e_1,Ae_2=e_1+e_2,\dots, Ae_n=e_1+\cdots+e_n\quad \text{with}\quad
A=
\begin{pmatrix}
1 & 1 & \cdots & 1\\
0 & 1 & \cdots & 1\\
\vdots & \vdots & \ddots & \vdots\\
0 & 0 & \cdots & 1\\
\end{pmatrix},
\]
because every $x=(x_i)_{i=1}^n\in W$ can be represented as
\[
x
=\sum_{i=1}^{n}(x_i-x_{i+1})(Ae_i)
=A\Big(\sum_{i=1}^{n}(x_i-x_{i+1})e_i\Big),
\]
where we set $x_{n+1}:=0$. A substitution gives
\[
\loou\cap W
=\Big\{x\in \IR^n\colon \sum_{i=1}^{n}i^{1/q-1}x_i\le 1, x=A\Big(\sum_{i=1}^{n}y_ie_i\Big),y_i\ge 0 \Big\}
=\Big\{Ay\in \IR^n\colon \sum_{i=1}^{n}\kappa_q(i)y_i\le 1, y_i\ge 0 \Big\}.
\]
Since $A$ is linear and bijective, it follows that the extreme points are given by $0$ and
$A(\kappa_q(i)^{-1}e_i)=Me_i$, $i\in \{1,\dots,n\}$.
\end{proof}

We now use this lemma to prove Theorem~\ref{thm:stochastic-rep}.

\begin{proof}[Proof of Theorem~\ref{thm:stochastic-rep}]
From Lemma~\ref{lem:extreme}, we know the extreme points of $\loou\cap W$, which are given by the images under $M$ of $0$ and the unit vectors $e_1,\dots,e_n$; $M$ is given as in Lemma \ref{lem:extreme}. Thus
\begin{equation} \label{eq:convex-hull}
\loou\cap W
= {\rm conv}\big\{0,Me_1,\dots,Me_n\big\}
=M\big({\rm conv}\big\{0,e_1,\dots,e_n\big\}\big).
\end{equation}
The convex hull ${\rm conv}\big\{0,e_1,\dots,e_n\big\}$
is the projection of the standard $n$-dimensional simplex $\Delta_n$ in $\IR^{n+1}$, i.e., of
\[
\Delta_{n}:=\Bigg\{x=(x_i)_{i=1}^{n+1}\in\IR^{n+1}\,:\,x_i \geq 0,\,\sum_{i=1}^{n+1}x_i = 1\Big\} = {\rm conv}\big\{ e_1,\dots,e_{n+1}\bigg\}\subseteq \IR^{n+1},
\]
onto the first $n$ coordinates. It is well known (see, e.g., \cite{MathaiBook}) that
\[
Z^{(n)}:=\Big(\frac{E_1}{\sum_{j=1}^{n+1}E_j},\dots,\frac{E_{n+1}}{\sum_{j=1}^{n+1}E_j}\Big)
\]
is uniformly distributed on $\Delta_{n}$, where $E_1,\dots,E_{n+1}$ are i.i.d.~standard exponential random variables. Applying the projection and then $M$ to $Z^{(n)}$, we see that by \eqref{eq:convex-hull} the random vector
\[
\Big(\frac{\sum_{j=1}^{n}\kappa_{q}(j)^{-1}E_j}{\sum_{j=1}^{n+1}E_j},\frac{\sum_{j=2}^{n}\kappa_{q}(j)^{-1}E_j}{\sum_{j=1}^{n+1}E_j},\dots,\frac{\sum_{j=n}^{n}\kappa_{q}(j)^{-1}E_j}{\sum_{j=1}^{n+1}E_j}\Big)
\]
is uniformly distributed on the intersection $\loou\cap W$. Using the 1-symmetry of $\loou$, we may change the order of coordinates (introducing a random permutation $\pi$ of $\{1,\dots,n\}$) as well as the signs of coordinates (introducing a random sign vector $(\varepsilon_1,\dots,\varepsilon_n)\in\{-1,1\}^n$) and obtain the desired probabilistic representation.
\end{proof}

We now present the proof of the asymptotics of the volume radius of a Lorentz ball. For this and also other proofs we shall need the following well-known asymptotics, see for example \cite[Lemma 3.4]{JP2022}.

\begin{lemma} \label{lem:sum-asymp}
Let $\alpha>-1$. There exists a constant $c_\alpha\in(0,\infty)$ and an absolute constant $c\in(0,\infty)$ such that for all $n\in\IN$,
\[
\Big|\sum_{i=1}^{n}i^{\alpha}-\frac{1}{\alpha+1}n^{\alpha+1}\Big|\le c_{\alpha}n^{\max\{0,\alpha\}} \qquad \text{ and }\qquad \Big|\sum_{i=1}^{n}i^{-1}-\log n\Big|\le c.
\]
\end{lemma}

\begin{proof}[Proof of Corollary~\ref{cor:volume}]
	From Lemma \ref{lem:extreme}, see Equation~\eqref{eq:convex-hull}, we know that $\loou\cap W=M({\rm conv}\{0,e_1,\dots,e_n\})$. Together with the $1$-symmetry, this leads to
	\begin{align*}
     \vol_n(\loou) & =2^nn! \cdot\vol_n\big(\loou\cap W\big)
     = 2^n \det(M)
     =2^n\prod_{i=1}^n \kappa_q(i)^{-1},
	\end{align*}
	where $W$ is as in \eqref{eq:weyl chamber} and $M$ as in Lemma~\ref{lem:extreme}. Now, observe that by Lemma~\ref{lem:sum-asymp} for a suitable constant $c_q\in(0,\infty)$ just depending on the parameter $q>1$ the following asymptotics hold:
	\begin{equation}
	\label{eq:kappa-asymp}
	|\kappa_q(i)-qi^{1/q}|\le c_q\,\,\text{ if }q<\infty \quad \text{and}\quad|\kappa_q(i)-\log(i+1)|\le c_{\infty} \,\,\text{ if }q=\infty.
\end{equation}
We start with case $q<\infty$. Here, we have
\begin{align}\label{eq:log product}
\log \Big(\prod_{i=1}^n \kappa_q(i)^{-1}\Big)^{1/n}
& =-\frac{1}{n}\sum_{i=1}^{n}\log \kappa_q(i) = -\frac{1}{n}\sum_{i=1}^{n}\log \Big(\frac{\kappa_q(i)\cdot qi^{1/q}}{qi^{1/q}}\Big) \cr
& =-\log q - \frac{1}{qn}\sum_{i=1}^{n}\log i
-\frac{1}{n}\sum_{i=1}^{n}\log\Big(\frac{\kappa_q(i)}{qi^{1/q}}\Big).
\end{align}
Using $\log(n!)= n\log n - n +O(\log n)$, we deduce that the first sum on the right-hand side satisfies, as $n\to\infty$,
\[
-\frac{1}{qn}\sum_{i=1}^{n}\log i
= - \frac{1}{q}\log(n) + \frac{1}{q} +O\Big(\frac{\log n}{n}\Big).
\]
By \eqref{eq:kappa-asymp} we have
\[
\lim_{i\to\infty}\log\Big(\frac{\kappa_q(i)}{qi^{1/q}}\Big)=0.
\]
By Cesàro's lemma,

\[
\lim_{n\to\infty}\frac{1}{n}\sum_{i=1}^{n}\log\Big(\frac{\kappa_q(i)}{qi^{1/q}}\Big)=0.
\]
This shows the desired result. If $q=\infty$, we proceed similarly but now use
\[
\frac{1}{n}\sum_{i=3}^{n}\log \log i
= \frac{1}{n}\int_3^n \log\log x + O\Big(\frac{1}{\log n}\Big)
=\log\log n- \frac{{\rm li}(n)}{n} +O\Big(\frac{1}{\log n}\Big)
= \log \log n+O\Big(\frac{1}{\log n}\Big),
\]
where the asymptotics of the logarithmic integral ${\rm li}(x):=\int_{0}^x\frac{\dd t}{\ln(t)}, x>1,$ are well known (\cite[Chapter 5]{AS64}). 
\end{proof}

\subsection{Proof of Theorems \ref{thm:main}}

We now continue with the proof of Theorem~\ref{thm:main}, i.e., the weak convergence results of the empirical distribution. In what follows, for two sequences $(a_n)_{n\in\IN}$ and $(b_n)_{n\in\IN}$ of non-negative real numbers, we shall write $a_n\lesssim b_n$ if there exists $C\in(0,\infty)$ such that for all $n\in\IN$ it holds that $a_n \leq C b_n$. If in addition $a_n \gtrsim b_n$, then we write $a_n \asymp b_n$. The following statement will be crucial in the proof.

\begin{lemma}\label{lem:uniform-app}
	Let $1<q\le \infty$ and for $n\in\N$ assume that $\widetilde{X}^{(n)}=(\widetilde{X}_1^{(n)},\dots,\widetilde{X}_n^{(n)})$ is uniformly distributed in $\loo$. Then, for any sequence $(a_n)_{n\in\IN}$ with $a_n\stackrel{n\to\infty}{\longrightarrow}\infty$, we have 
\[
\IP\Bigg[\sup_{1\le i\le n}\Big|(\widetilde{X}_i^{(n)})^*-G_q\Big(\frac{i-1}{n}\Big)\Big|\ge a_n\delta_{n}\Bigg]\stackrel{n\to\infty}{\longrightarrow} 0,
\]
where
\[
G_q\colon [0,1]\to [0,(q-1)^{-1}],\quad x\mapsto G_q(x)
=\frac{1}{q-1}(1-x^{1-1/q})\quad \text{if }q<\infty
\]
and $G_{\infty}\colon [0,1]\to [0,1]$, $x\mapsto 1-x$. The error bounds may be chosen as follows:
\[
\delta_{n}
=
\begin{cases}
	n^{-(1-1/q)}&: 1<q<2,\\
	n^{-1/2}\log n&: q=2,\\
	n^{-1/q}&: q\in (2,\infty),\\
	(\log n)^{-1}&: q=\infty,\\
\end{cases}\quad
\text{where }n\in\IN.
\]
\end{lemma}
\begin{proof}
Let $1<q<\infty$ and $i\in \{1,\dots,n\}$. Using the probabilistic representation from Theorem~\ref{thm:stochastic-rep}, we obtain for the $i^{th}$-largest entry $(\widetilde{X}_i^{(n)})^*$ of $(|\widetilde{X}_i^{(n)}|)_{i=1}^n$ that
\[
(\widetilde{X}_i^{(n)})^*
\stackrel{\ed}{=}Y_i^{(n)}\Big(\frac{1}{n}\sum_{j=1}^{n+1}E_j\Big)^{-1}\quad\text{with}\quad
Y_i^{(n)}=\frac{1}{n}\sum_{j=i}^{n}\frac{n^{1/q}}{\kappa_q(j)}E_j,
\]
and the same holds in the sense of joint distributions. We will prove the statement for the random variable $Y_i^{(n)}$ instead of $(\widetilde{X}_i^{(n)})^*$ and then the statement of the lemma easily follows, because as a consequence of the triangle inequality and the sub-additivity of the supremum, we have
\begin{align*}
&\IP\Bigg[\sup_{1\le i\le n}\Big|(\widetilde{X}_i^{(n)})^*-G_q\Big(\frac{i-1}{n}\Big)\Big|\ge a_n\delta_n\Bigg] \cr
& \leq  \IP\Bigg[\sup_{1\le i\le n}\Big|(\widetilde{X}_i^{(n)})^*-Y_i^{(n)}\Big|\ge \frac{a_n}{2}\delta_n\Bigg] + \IP\Bigg[\sup_{1\le i\le n}\Big|Y_i^{(n)}-G_q\Big(\frac{i-1}{n}\Big)\Big|\ge \frac{a_n}{2}\delta_n\Bigg].
\end{align*}
In order to estimate the first probability on the right-hand side, we make use of the fact that
\begin{equation}\label{eq: sup equality}
\sup_{1\le i\le n}\Big|(\widetilde{X}_i^{(n)})^*-Y_i^{(n)}\Big|
\stackrel{\ed}{=} \big(\sup_{1\leq i \leq n} Y_i^{(n)}\big)\cdot\Big|\Big(\frac{1}{n}\sum_{j=1}^{n+1}E_j\Big)^{-1}-1\Big|
= \|Y^{(n)}\|_{\infty}\cdot\Big|\Big(\frac{1}{n}\sum_{j=1}^{n+1}E_j\Big)^{-1}-1\Big|.
\end{equation}

From the proof below it will follow that the first factor on the right-hand side of \eqref{eq: sup equality} is bounded by a suitable constant with probability tending to one. The second factor satisfies (via the central limit theorem)
\begin{align*}
\IP\Bigg[\Big|\Big(\frac{1}{n}\sum_{j=1}^{n+1}E_j\Big)^{-1}-1\Big|  \ge \delta_n \Bigg]
& \leq \IP\Bigg[\frac{1}{n}\sum_{j=1}^{n+1}E_j\le (1+\delta_n)^{-1} \Bigg]
 \leq \IP\Bigg[\frac{1}{n}\sum_{j=1}^{n}E_j\le (1+\delta_n)^{-1} \Bigg] \cr
& =\IP\Bigg[\frac{s_n}{n}\sum_{j=1}^{n}(E_j-1)\le -1 \Bigg] = \IP\Bigg[\frac{1}{\sqrt{n}}\sum_{j=1}^{n}(E_j-1)\le -\frac{\sqrt{n}}{s_n} \Bigg]
\stackrel{n\to\infty}{\longrightarrow} 0,
\end{align*}
where
$s_n:= 1 + \delta_n^{-1} \asymp \delta_n^{-1}$ satisfies $\sqrt{n}/s_n \to +\infty$, as a consequence of the definition of $\delta_n$ stated in the lemma.  Taking into account \eqref{eq: sup equality} and the upper bound on the largest coordinate of $Y^{(n)}$, one easily derives the statement for $\widetilde{X}^{(n)}$ instead of $Y^{(n)}$.

We now start with the proof of the statement for $Y_i^{(n)}$. Consider the decomposition
\begin{equation}\label{eq:three-sums}
	Y^{(n)}_i-G_q\Big(\frac{i-1}{n}\Big)
=\frac{1}{n}\sum_{j=i}^{n}\frac{n^{1/q}}{\kappa_q(j)}E_j-\frac{1}{q}\int_{(i-1)/n}^1 x^{-1/q}\,\dd x
= \frac{1}{n}\sum_{j=i}^{n}A_j +\frac{1}{n}\sum_{j=i}^{n}B_j +\frac{1}{qn}\sum_{j=i}^{n} C_j
\end{equation}
with summands
\[
A_j=\frac{n^{1/q}}{\kappa_q(j)}(E_j-1),\quad
B_j=\frac{n^{1/q}}{\kappa_q(j)}-\frac{n^{1/q}}{qj^{1/q}},\quad
C_j=\Big(\frac{j}{n}\Big)^{-1/q}-\int_{j-1}^{j} \Big(\frac{x}{n}\Big)^{-1/q}\,\dd x, \quad j=i,\dots,n.
\]
Here we used $G_q(x) =\frac{1}{q}\int_{x}^1 y^{-1/q}\,\dd y$, $x\in [0,1]$. Only the first sum on the right-hand side of \eqref{eq:three-sums} is random and its variance is
\[
\frac{1}{n^2}\sum_{j=i}^{n}{\rm Var}(A_j)
=\frac{1}{n^2}\sum_{j=i}^{n}\Big(\frac{n^{1/q}}{\kappa_q(j)}\Big)^2
\lesssim_q n^{-2(1-1/q)}\sum_{j=1}^{n}j^{-2/q}
\lesssim_q
\begin{cases}
	n^{-2(1-1/q)}&\text{if }1<q<2,\\
	n^{-1}\log n&\text{if }q=2,\\
	n^{-1}&\text{if }q>2,\\
\end{cases}
\]
which tends to zero as $n\to\infty$. Here we used the asymptotics $\kappa_q(j)\asymp_q j^{1/q}$ and sum-integral approximation, see \eqref{eq:kappa-asymp} and Lemma~\ref{lem:sum-asymp}.

For the second sum we have due to \eqref{eq:kappa-asymp}
\[
|B_j|=\frac{n^{1/q}}{qj^{1/q}}\frac{|qj^{1/q}-\kappa_q(j)|}{\kappa_q(j)}
\asymp_q \frac{n^{1/q}}{j^{1/q}}\frac{1}{j^{1/q}}
=n^{1/q}j^{-2/q},
\]
which gives that
\[
\frac{1}{n}\sum_{j=i}^{n}|B_j|
\lesssim_q n^{1/q-1}\sum_{j=1}^{n}j^{-2/q}
\lesssim_q
\begin{cases}
	n^{-(1-1/q)}&\text{if }1<q<2,\\
	n^{-1/2}\log n&\text{if }q=2,\\
	n^{-1/q}&\text{if }q>2.\\
\end{cases}
\]
For the third sum we use that for any $j\ge 1$, we find $\xi\in (j-1,j)$ such that
\[
0 \le \int_{j-1}^{j}\Big(\frac{x}{n}\Big)^{-1/q}\dd x - \Big(\frac{j}{n}\Big)^{-1/q}
= \Big(\frac{\xi}{n}\Big)^{-1/q} - \Big(\frac{j}{n}\Big)^{-1/q}
\lesssim \Big(\frac{j}{n}\Big)^{-1/q}\frac{1}{j},
\]
which follows from the mean value theorem used twice. Thus, the third sum satisfies
\[
\frac{1}{n}\sum_{j=i}^{n}|C_j|
\lesssim_q n^{-(1-1/q)}\sum_{j=1}^{n}j^{-(1+1/q)}
\lesssim_q n^{-1+1/q}.
\]
Lastly, we note that
\begin{align*}
\IP\Bigg[\sup_{1\le i\le n}\Big|Y_i^{(n)}-G_q\Big(\frac{i-1}{n}\Big)\Big|\ge a_n\delta_n\Bigg]
&\le \IP\Bigg[\sup_{1\le i\le n}\Big|\frac{1}{n}\sum_{j=i}^{n}A_j\Big|\ge \frac{a_n}{3}\delta_n\Bigg]\cr
&+ \IP\Bigg[\sup_{1\le i\le n}\Big|\frac{1}{n}\sum_{j=i}^{n}B_j\Big|\ge \frac{a_n}{3}\delta_n\Bigg] \cr
&+ \IP\Bigg[\sup_{1\le i\le n}\Big|\frac{1}{n}\sum_{j=i}^{n}C_j\Big|\ge \frac{a_n}{3}\delta_n\Bigg].
\end{align*}
For large $n$ the last two summands vanish because of the previous estimates and the first tends to zero due to Kolmogorov's inequality (\cite[Thm. 5.28]{Kle08}). This proves the statement for $q<\infty$.

If $q=\infty$ we need to carry out several modifications, in particular adapt the definition of $Y^{(n)}$ by replacing $n^{1/q}$ with $\log(n+1)$. Let us write for $i\in \{1,\dots,n\}$ again the coordinate as in equation \eqref{eq:three-sums}, i.e.,
\[
Y^{(n)}_i-G_{\infty}\Big(\frac{i-1}{n}\Big)
=\frac{1}{n}\sum_{j=i}^{n}\frac{\log(n+1)}{\kappa_{\infty}(j)}E_j-\Big(1-\frac{i-1}{n}\Big)
= \frac{1}{n}\sum_{j=i}^{n}A_j +\frac{1}{n}\sum_{j=i}^{n}B_j +\frac{1}{n}\sum_{j=i}^{n} C_j
\]
with summands
\[
A_j=\frac{\log(n+1)}{\kappa_{\infty}(j)}(E_j-1),\quad
B_j=\frac{\log(n+1)}{\kappa_{\infty}(j)}-\frac{\log(n+1)}{\log(j+1)},\quad
C_j=\frac{\log(n+1)}{\log(j+1)}-1.
\]
One can estimate the errors similarly after adjusting $\delta_n$. For the first sum we have $\kappa_{\infty}(j)\asymp \log(j+1)$ using \eqref{eq:kappa-asymp} and thus
\[
\frac{1}{n^2}\sum_{j=i}^{n}{\rm Var}(A_j)
=\frac{1}{n^2}\sum_{j=i}^{n}\Big(\frac{\log(n+1)}{\kappa_{\infty}(j)}\Big)^2
\lesssim \frac{\log^2(n+1)}{n^2}\sum_{j=1}^{n}\frac{1}{\log^2(j+1)}
\lesssim \frac{1}{n},
\]
where in the last estimate we used that $\log^2(n+1)$ is slowly varying and thus the arithmetic mean is equivalent to the largest summand, see also \cite[Prop. 1.5.8]{BGT87}. Similarly, we have
\[
\frac{1}{n}\sum_{j=i}^{n}|B_j|
\lesssim \frac{\log(n+1)}{n}\sum_{j=1}^{n}\frac{1}{\log^2(j+1)}
\lesssim \frac{1}{\log(n+1)}.
\]
The third sum satisfies
\[
\frac{1}{n}\sum_{j=i}^{n}|C_j|
\le\frac{1}{n}\sum_{j=1}^{n}|C_j|
=\frac{\log(n+1)}{n}\sum_{j=1}^{n}\frac{1}{\log(j+1)}-1
\lesssim \frac{1}{\log(n+1)},
\]
where we used the approximation
\[
\sum_{j=1}^{n}\frac{1}{\log(j+1)}
=\int_{2}^{n+1} \frac{1}{\log x}\,\dd x+O(1)
= {\rm li}(n+1) +O(1)
= \frac{n}{\log(n+1)}+O\Big(\frac{n}{\log^2(n+1)}\Big).
\]
The proof is now completed as in the case $q<\infty$.
\end{proof}


For the proof of Theorem~\ref{thm:main} we use a symmetrization trick, which we present in the following lemma; the result may well be known, but we were unable to find it in the literature. 

\begin{lemma} \label{lem:symmetry}
	For each $n\in\IN$ let $X^{(n)}=(X_1^{(n)},\dots,X_n^{(n)})$ be a random vector in $\IR^n$ with identically distributed coordinates such that $(\varepsilon_1 X_1^{(n)},\dots,\varepsilon_n X_n^{(n)})\overset{\ed}{=}X^{(n)}$ for any choice of signs $(\varepsilon_1,\dots,\varepsilon_n)\in \{-1,1\}^n$. If for every bounded continuous function $f\colon [0,\infty)\to\IR$
\[
\frac{1}{n}\sum_{i=1}^{n}f\big(|X^{(n)}_i|\big) \xrightarrow[n\to\infty]{\IP} \int_{0}^{\infty}f(x)\varphi(x)\, \dd x,
\]
holds for a density $\varphi\colon [0,\infty)\to [0,\infty)$, then for every bounded continuous function $f\colon \IR\to\IR$
\[
\frac{1}{n}\sum_{i=1}^{n}f\big(X^{(n)}_i\big) \xrightarrow[n\to\infty]{\IP}\frac{1}{2}\int_{-\infty}^{\infty}f(x)\varphi(|x|)\, \dd x.
\]
\end{lemma}
\begin{proof}

Let $f\colon \IR\to\IR$ be a bounded and continuous function. The two mappings
\[
f_1\colon [0,\infty)\to\IR,\quad x\mapsto f(x)
\quad \text{and}\quad f_{-1}\colon [0,\infty)\to\IR,\quad x\mapsto f(-x)
\]
are again bounded and continuous on $[0,\infty)$. Let $\varepsilon_1,\dots,\varepsilon_n$ be independent Rademacher random variables, independent of all other random objects. Then the symmetry of $X^{(n)}$ implies
\[
S_n
:=\frac{1}{n}\sum_{i=1}^{n}f\big(X^{(n)}_i\big)
\overset{\rm d}{=}\frac{1}{n}\sum_{i=1}^{n}f\big(\varepsilon_i |X^{(n)}_i|\big)
=\frac{1}{n}\sum_{i=1}^{n}f_{\varepsilon_i}\big(|X^{(n)}_i|\big).
\]
By linearity of the conditional expectation, we obtain 
\begin{align}\label{eq:cond-exp}
\IE\big[S_n\big| X^{(n)}\big]
&= \frac{1}{n}\sum_{i=1}^{n}\IE\big[f_{\varepsilon_i}\big(|X^{(n)}_i|\big)\big| X^{(n)}\big]\nonumber\\
&=\frac{1}{n}\sum_{i=1}^{n}\frac{(f_1+f_{-1})(|X^{(n)}_i|)}{2}
\xrightarrow[n\to\infty]{\IP} \int_{0}^{\infty}\frac{(f_{1}+f_{-1})(x)}{2}\varphi(x)\, \dd x
=\frac{1}{2}\int_{-\infty}^{\infty}f(x)\varphi(|x|)\, \dd x.
\end{align}
Further, the conditional variance satisfies
\[
{\rm Var}\big[S_n\big| X^{(n)}\big]
=\frac{1}{2n^2}\sum_{i=1}^{n}\bigg(f_1\big(|X_i^{(n)}|\big)-\frac{(f_1+f_{-1})(|X_i^{(n)}|)}{2}\bigg)^2+\bigg(f_{-1}\big(|X_i^{(n)}|\big)-\frac{(f_{1}+f_{-1})(|X_i^{(n)}|)}{2}\bigg)^2
\le \frac{4\|f\|_{\infty}^2}{n}.
\]
Conditioned on $X^{(n)}$, Chebyshev's inequality gives, for every $\varepsilon>0$,
\[
\IP\big[|S_n-\IE[S_n\big| X^{(n)}]|>\varepsilon \big| X^{(n)}\big]
\le \frac{ {\rm Var}\big[S_n\big| X^{(n)}\big]}{\varepsilon^2}.
\]
Thus, the law of total expectation yields
\[
\IP\big[|S_n-\IE[S_n\big| X^{(n)}]|>\varepsilon \big]
\le \frac{ \IE\big[\,{\rm Var}[S_n\big| X^{(n)}]\big]}{\varepsilon^2}
\le \frac{4\|f\|_{\infty}^2}{\varepsilon^2 n} \stackrel{n\to\infty}{\longrightarrow} 0.
\]
This together with \eqref{eq:cond-exp} implies
\[
S_n \xrightarrow[n\to\infty]{\IP} \frac{1}{2}\int_{-\infty}^{\infty}f(x)\varphi(|x|)\, \dd x,
\]
which completes the proof.
\end{proof}

\begin{proof}[Proof of Theorem~\ref{thm:main}]

We use Lemma \ref{lem:symmetry}, which allows us to pass to the absolute values. Let $\varepsilon>0$ and $f\colon [0,\infty)\to [0,\infty)$ be bounded and continuous. We have
\begin{align}\label{eq:decomposition}
\IP\Bigg[\Big|\frac{1}{n}\sum_{i=1}^{n}f\big(|\widetilde{X}_i^{(n)}|\big)-\int_{\IR_+}f(x)\, \nu_{q,1}(\dd x)\Big|>\varepsilon\Bigg]
&\le
\IP\Bigg[\Big|\frac{1}{n}\sum_{i=1}^{n}f\big(|\widetilde{X}_i^{(n)}|\big)-\frac{1}{n}\sum_{i=1}^{n}f\Big(G_q\Big(\frac{i-1}{n}\Big)\Big)\Big|>\frac{\varepsilon}{2}\Bigg]\nonumber\\
&+\IP\Bigg[\Big|\frac{1}{n}\sum_{i=1}^{n}f\Big(G_q\Big(\frac{i-1}{n}\Big)\Big)-\int_{\IR_+}f(x)\, \nu_{q,1}(\dd x)\Big|>\frac{\varepsilon}{2}\Bigg].
\end{align}

For the first summand on the right-hand side, we note that
\[
\Big|\frac{1}{n}\sum_{i=1}^{n}f\big(|\widetilde{X}_i^{(n)}|\big)-\frac{1}{n}\sum_{i=1}^{n}f\Big(G_q\Big(\frac{i-1}{n}\Big)\Big)\Big|
\le \sup_{1\leq i\leq n}\Big|f\big((\widetilde{X}_i^{(n)})^*\big)-f\Big(G_q\Big(\frac{i-1}{n}\Big)\Big)\Big|,
\]
since we may arrange the absolute values in any order without changing the average. Let $(a_n)_{n\in\IN}$ be a sequence with $a_n\stackrel{n\to\infty}{\longrightarrow}\infty$ such that $a_n\delta_n\stackrel{n\to\infty}{\longrightarrow} 0$ and $a_n\delta_n\le 1$ for all $n\in\IN$. For $n\in\IN$, we define the event
\[
A_n:=\Bigg\{\sup_{1\leq i\leq n}\Big|(\widetilde{X}_i^{(n)})^*-G_q\Big(\frac{i-1}{n}\Big)\Big|\le a_n\delta_n\Bigg\}.
\]
Then on $A_n$ it holds that: $(\widetilde{X}_i^{(n)})^*\le (\widetilde{X}_1^{(n)})^*\le G_q(0)+1$ for all $i\in\{1,\dots,n\}$  and
\[
\sup_{1\leq i\leq n}\Big|f\big((\widetilde{X}_i^{(n)})^*\big)-f\Big(G_q\Big(\frac{i-1}{n}\Big)\Big)\Big|
\le \sup_{\substack{x,y\in [0,G_q(0)+1]\\ |x-y|\le a_n\delta_n}} |f(x)-f(y)|,
\]
which tends to zero because $f$ is uniformly continuous on $[0,G_q(0)+1]$. By Lemma~\ref{lem:uniform-app} it holds that $\IP[A_n]\stackrel{n\to\infty}{\longrightarrow} 1$, and thus the first summand in \eqref{eq:decomposition} tends to zero as well. For the second summand in \eqref{eq:decomposition}, we note that the continuity of $f$ and $G_q\colon [0,1]\to [0,\infty)$ imply via a substitution that
\[
\frac{1}{n}\sum_{i=1}^{n}f\Big(G_q\Big(\frac{i-1}{n}\Big)\Big)
\stackrel{n\to\infty}{\longrightarrow} \int_0^1 f(G_q(x))\,\dd x
=\int_0^{G_q(0)}f(x)(-G_q^{-1})'(x)\,\dd x,
\]
where
\[
G_q^{-1}\colon [0,G_q(0)],\quad x\mapsto (1-(q-1)x)^{q/(q-1)},
\]
if $q<\infty$ and $G_{\infty}^{-1}=G_{\infty}$, such that $(-G_q^{-1})'$ is the density of $\nu_{q,1}$.
\end{proof}

With Theorem~\ref{thm:main} at our disposal, we are able to deduce Corollary~\ref{cor:coordinates} on the asymptotic distribution of a single coordinate.

\begin{proof}[Proof of Corollary~\ref{cor:coordinates}]
	The result essentially follows from the exchangeability of the coordinates of $\widetilde{X}^{(n)}=(\widetilde{X}_1^{(n)},\dots,\widetilde{X}_n^{(n)})$ and in mathematical physics literature is known as propagation of chaos, see, e.g., the work \cite{Szn91} of Sznitman; recall that exchangeability means that any permutation of coordinates has the same joint distribution as the original one. For convenience of the reader, we adapt an argument from \cite[p. 326]{DZ10}.
	
Let $k\in\IN$. We show for any bounded and continuous function $f\colon \IR^k\to\IR$ that
\[
\IE \big[f(\widetilde{X}_1^{(n)},\dots,\widetilde{X}_k^{(n)}) \big] \stackrel{n\to\infty}{\longrightarrow} \big(\IE f(Y)\big)^k,
\]
where $Y$ is real random variable distributed according to the Lebesgue density $f_{q,1}$. 
It is sufficient to assume that $f=\prod_{i=1}^k f_i$ for $f_i\colon \IR\to\IR$ bounded and continuous (see \cite[Appendix D, p. 356]{DZ10}). Due to exchangeability and linearity, we have
\[
\IE \Big[\prod_{i=1}^k f_i(\widetilde{X}_i^{(n)})\Big]
=\IE\Bigg[ \frac{(n-k)!}{n!} \sum_{i_1\neq \cdots \neq i_k}\prod_{i=1}^k f_i(\widetilde{X}_{i_j}^{(n)})\Bigg].
\]
Let $L_{\widetilde{X}^{(n)}}=\frac{1}{n}\sum_{i=1}^{n}\delta_{\widetilde{X}^{(n)}_i}$. Then we also have that
\[
\IE \Bigg[ \int \prod_{i=1}^k f_i \,\dd L_{\widetilde{X}^{(n)}}^{\otimes k} \Bigg]
=\IE \Bigg[ \prod_{i=1}^k \frac{1}{n}\sum_{j=1}^{n}f_i(\widetilde{X}_j^{(n)}) \Bigg]
=\frac{1}{n^k}\IE \Bigg[ \sum_{i_1,\dots,i_k}\prod_{i=1}^k f_i(\widetilde{X}_{i_j}^{(n)})\Bigg].
\]
Therefore, it holds that
\[
\Big|\IE \prod_{i=1}^k f_i(\widetilde{X}_i^{(n)}) -\IE \int \prod_{i=1}^k f_i \dd L_{\widetilde{X}^{(n)}}^{\otimes k}\Big|
=\Big|1-\frac{n!}{(n-k)!n^k}\Big|\max_{1\leq i \leq k}\|f_i\|_{\infty}^k\stackrel{n\to\infty}{\longrightarrow} 0.
\]
Applied to the $f_i$'s, Theorem~\ref{thm:main} and the continuous mapping theorem give the convergence
\[
\int \prod_{i=1}^k f_i \,\dd L_{\widetilde{X}^{(n)}}^{\otimes k}
=\prod_{i=1}^k\int  f_i \,\dd L_{\widetilde{X}^{(n)}}
\xrightarrow[n\to\infty]{\mathbb{P}} \prod_{i=1}^k f_i(Y).
\]
By uniform boundedness we can take expectations on both sides and this completes the proof.
\end{proof}



\subsection{Proof of Theorem \ref{thm:clt-inf}}

We now prove the central limit theorem for the maximum norm of a random vector in a Lorentz ball.

\begin{proof}[Proof of Theorem~\ref{thm:clt-inf}]
We have
\[
\|\widetilde{X}^{(n)}\|_{\infty}
\overset{\ed}{=} \frac{Y_n}{Z_n}\quad\text{with}\quad Y_n:=\frac{1}{n}\sum_{j=1}^{n}\frac{n^{1/q}}{\kappa_q(j)}E_j\quad\text{and}\quad Z_n:=\frac{1}{n}\sum_{j=1}^{n+1}E_j.
\]

We first prove (i). Let $1\le q<2$ and write
\[
Y_n
=n^{1/q-1}\sum_{j=1}^{n}\frac{E_j-1}{\kappa_q(j)}+\mu_{q,n} \quad \text{with} \quad \mu_{q,n}=\frac{1}{n}\sum_{j=1}^{n}\frac{n^{1/q}}{\kappa_q(j)}
\]
as in the statement of the theorem. Note that
\[
\IE\Big[ \frac{E_j-1}{\kappa_q(j)}\Big]
=0
\quad \text{and} \quad
\Var\Big[\frac{E_j-1}{\kappa_q(j)}\Big]
= \frac{1}{\kappa_q(j)^2}.
\]
Since $\kappa_q(j)\asymp_q j^{1/q}$, we have
\[
\sum_{j=1}^{\infty}\Var\Big[\frac{E_j-1}{\kappa_q(j)}\Big]
=\sum_{j=1}^{\infty}\frac{1}{\kappa_q(j)^2}
<\infty
\]
and thus by the martingale convergence theorem (\cite[Thm. 11.4]{Kle08})
\begin{equation} \label{eq:martingale-conv}
\widetilde{Y}_n
:=n^{1-1/q}Y_n-\sum_{j=1}^{n}\frac{1}{\kappa_q(j)}
=\sum_{j=1}^{n}\frac{E_j-1}{\kappa_q(j)}\xrightarrow[n\to\infty]{\rm a.s.} \sum_{j=1}^{\infty}\frac{E_j-1}{\kappa_q(j)}=:R_q.
\end{equation}
Write
\[
n^{1-1/q}\big(\|\widetilde{X}^{(n)}\|_{\infty}-\mu_{q,n}\big)
\overset{\ed}{=}n^{1-1/q}\Big(\frac{Y_n}{Z_n}-\mu_{q,n}\Big)
=\frac{\widetilde{Y}_n+\sum_{j=1}^{n}\frac{1}{\kappa_q(j)}}{1+n^{-1/2}\widetilde{Z}_n+n^{-1}}-\sum_{j=1}^{n}\frac{1}{\kappa_q(j)},
\]
where $\widetilde{Z}_n:=n^{-1/2}\sum_{j=1}^{n+1}(E_j-1)\stackrel{\ed}{\longrightarrow} \mathscr{N}(0,1)$ by the central limit theorem and Slutsky's theorem (\cite[Thm. 13.18]{Kle08}) since $E_{n+1}/\sqrt{n}\stackrel{\mathbb{P}}{\longrightarrow} 0$. Simplifying, by Slutsky's theorem,
\[
n^{1-1/q}\big(\|\widetilde{X}^{(n)}\|_{\infty}-\mu_{q,n}\big)
\overset{\ed}{=}\frac{\widetilde{Y}_n-n^{-1/2}\Big(\sum_{j=1}^{n}\frac{1}{\kappa_q(j)}\Big)(\widetilde{Z}_n+n^{-1/2})}{1+n^{-1/2}\widetilde{Z}_n+n^{-1}}\xrightarrow[n\to\infty]{\ed} R_q,
\]
since $n^{-1/2}\Big(\sum_{j=1}^{n}\frac{1}{\kappa_q(j)}\Big)\stackrel{n\to\infty}{\longrightarrow} 0$.

For convenience of the reader we prove that $R_1$ is Gumbel distributed (see, e.g., \cite[Theorem 1.1 (c)]{KPT2019_I}). It is well-known that $\sum_{j=1}^{n}\frac{1}{j}=\log n +\gamma +o(1)$ and from Remark~\ref{rem:one}, we deduce that
\[
\sum_{j=1}^{n}\frac{E_j-1}{j}
\overset{\ed}{=} \max_{1\leq j\leq n } E_j -\log n- \gamma+e_n,
\]
where $e_n \stackrel{n\to\infty}{\longrightarrow} 0$. Now the well-known fact that $\max_{1\leq j \leq n } E_j -\log n\xrightarrow[n\to\infty]{\ed} G$, where $G$ is standard Gumbel distributed, Slutsky's theorem and uniqueness of the limit imply that $R_1\overset{\ed}{=} G-\gamma$.

For the proof of (ii) let $q=2$. Defining $\widetilde{Y}_n$ as on the left-hand side of \eqref{eq:martingale-conv}, by a version of the Lindeberg central limit theorem from \cite[Thm. 5.3]{Das08}, we have
\[
\frac{1}{\sqrt{\log n}}\widetilde{Y}_n
=\frac{1}{\sqrt{\log n}}\sum_{j=1}^{n}\frac{E_j-1}{\kappa_2(j)}\xrightarrow[n\to\infty]{\ed} \mathscr{N}(0,1/4),
\]
since $\sum_{j=1}^{n}\frac{1}{\kappa_2(j)^2}\sim \frac{1}{4}\log n \stackrel{n\to\infty}{\longrightarrow} \infty$.

A similar rewriting and again Slutsky's theorem give
\[
\frac{\sqrt{n}}{\sqrt{\log n}}\big(\|\widetilde{X}^{(n)}\|_{\infty}-\mu_{2,n}\big)
=\frac{\frac{\widetilde{Y}_n}{\sqrt{\log n}}-\frac{n^{-1/2}}{\sqrt{\log n}}\Big(\sum_{j=1}^{n}\frac{1}{\kappa_2(j)}\Big)(\widetilde{Z}_n+1)}{1+n^{-1/2}\widetilde{Z}_n+n^{-1/2}}\xrightarrow[n\to\infty]{\ed} \mathscr{N}(0,1/4),
\]
since $\frac{n^{-1/2}}{\sqrt{\log n}}\Big(\sum_{j=1}^{n}\frac{1}{\kappa_2(j)}\Big)\asymp \frac{1}{\sqrt{\log n}}\stackrel{n\to\infty}{\longrightarrow} 0$.

In order to prove (iii), let now $q>2$. We will show that the vector $\sqrt{n}(Y_n-\mu_{q,n},Z_n-1)$ tends to a Gaussian random vector and then use Taylor's theorem to conclude the proof. We have
\[
\sqrt{n}(Y_n-\mu_{q,n},Z_n-1)
=\frac{1}{\sqrt{n}}\sum_{j=1}^{n}\Big(\frac{n^{1/q}}{\kappa_q(j)}(E_j-1),E_j-1\Big)+\Big(0,\frac{E_{n+1}}{\sqrt{n}}\Big).
\]
Note that, since $\frac{E_{n+1}}{\sqrt{n}}$ tends to zero in probability, by Slutsky's theorem we can omit the last summand as we are dealing with convergence in distribution.  For the sum of random vectors on the right-hand side we check Lyapunov's condition for the multivariate central limit theorem (see, e.g., 
\cite[Thm. 3.2.2]{MS01}). To this end, let $\delta>0$ with $q>2+\delta$ and compute
\begin{align*}
n^{-(2+\delta)/2}\sum_{i=1}^{n}\IE\Big|\frac{n^{1/q}}{\kappa_q(j)}E_j\Big|^{2+\delta}
&\lesssim n^{-(2+\delta)/2}\sum_{i=1}^{n}\Big|\frac{n^{1/q}}{\kappa_q(j)}\Big|^{2+\delta}\\
&\lesssim_q n^{(2+\delta)(-1/2+1/q)}\sum_{i=1}^{n}j^{-(2+\delta)/q}
\lesssim_q n^{1-(2+\delta)/2}\stackrel{n\to\infty}{\longrightarrow} 0.
\end{align*}
It remains to compute the limit of the covariance matrix. We have $\Var\Big[\frac{1}{\sqrt{n}}\sum_{j=1}^{n}E_j\Big]=1$ for all $n\in\IN$ as well as
\[
\lim_{n\to\infty}\Var\Big[\frac{1}{\sqrt{n}}\sum_{j=1}^{n}\frac{n^{1/q}}{\kappa_q(j)}E_j\Big]
=\lim_{n\to\infty}\frac{1}{n}\sum_{j=1}^{n}\frac{n^{2/q}}{\kappa_q(j)^2}
=\frac{1}{q^2}\int_0^1 x^{-2/q}\dd x
=\frac{1}{q(q-2)}
=:s_q^2,
\]
where we used that $\kappa_q(j)\sim qj^{1/q}$ as $j\to\infty$. Similarly, 
\[
\lim_{n\to\infty}\frac{1}{n}\sum_{j=1}^{n}\Cov\Big[\frac{n^{1/q}}{\kappa_q(j)}E_j,E_j\Big]
=\lim_{n\to\infty}\frac{1}{n}\sum_{j=1}^{n}\frac{n^{1/q}}{\kappa_q(j)}
=\frac{1}{q}\int_0^{1}x^{-1/q}\dd x
=\frac{1}{q-1}=:\mu_{q,\infty}.
\]
Therefore, the limiting covariance matrix is
\[
\Sigma
:=\lim_{n\to\infty}\Cov\big[\sqrt{n}(Y_n-\mu_{q,n},Z_n-1)\big]
=\begin{pmatrix}
	s_q^2 & \mu_{q,\infty}\\
	\mu_{q,\infty} & 1\\
\end{pmatrix}
\]
and we derive the central limit theorem
\begin{equation} \label{eq:clt}
\sqrt{n}(Y_n-\mu_{q,n},Z_n-1)\xrightarrow[n\to\infty]{\ed} (Y,Z)\sim \mathscr{N}(0,\Sigma).
\end{equation}
In order to prove a central limit theorem for $\frac{Y_n}{Z_n}$, note that the function $(x,y)\mapsto F(x,y)=\frac{x}{y}$ is continuously differentiable for $x,y>0$ and $\nabla F(x,y)=(\frac{1}{y},-\frac{x}{y^2})^{\top}$. Hence, by Taylor's theorem,
\begin{equation} \label{eq:taylor}
\sqrt{n}\big(F(Y_n,Z_n)-F(\mu_{q,n},1)\big)
= \sqrt{n}\big((Y_n,Z_n)-(\mu_{q,n},1)\big)\nabla F(\mu_{q,n},1)^{\top}+e_{n,q},
\end{equation}
where the random error term is
\[
e_{n,q}:=\|\sqrt{n}(Y_n-\mu_{q,n},Z_n-1)\|_2h(Y_n-\mu_{q,n},Z_n-1),
\]
where $h$ is a function tending to zero as its argument approaches zero. Due to the CLT in \eqref{eq:clt} the random variable $\|\sqrt{n}(Y_n-\mu_{q,n},Z_n-1)\|_2$ stays bounded in probability and $h(Y_n-\mu_{q,n},Z_n-1)$ tends to zero in probability. Thus, the error satisfies $e_{n,q}\xrightarrow[n\to\infty]{\IP} 0$. For the first term on the right-hand side of \eqref{eq:taylor}, we note that $\nabla F(\mu_{q,n},1)^{\top}\stackrel{n\to\infty}{\longrightarrow}\nabla F(\mu_{q,\infty},1)^{\top}$. Therefore, by \eqref{eq:clt} and \eqref{eq:taylor} together with Slutsky's theorem, we have
\[
\sqrt{n}(F(Y_n,Z_n)-F(\mu_{q,n},1)) \xrightarrow[n\to\infty]{\ed}  (Y,Z)\nabla F(\mu_{q,\infty},1)^{\top} \sim \mathscr{N}(0,\sigma_q^2)
\]
with
\[
\sigma_q^2
:=(1,-\mu_{q,\infty})\Sigma (1,-\mu_{q,\infty})^{\top}
=\frac{1}{q(q-1)^2(q-2)}.
\]
This implies the claimed central limit theorem for $\|\widetilde{X}^{(n)}\|_{\infty}$.
\end{proof}

\subsection{Proof of Theorem \ref{pro:lln-norm}}

We now present the proof of Theorem \ref{pro:lln-norm}, i.e., of the weak law of large numbers for the $\ell_r^n$ norm of random points in normalized Lorentz balls. A key ingredient is again Lemma \ref{lem:uniform-app}.

\begin{proof}[Proof of Theorem~\ref{pro:lln-norm}]
	Let $r<\infty$ and assume first that $\widetilde{X}^{(n)}$ is uniformly distributed on $\loo$. First, because of the permutation invariance of the $\ell_r^n$ norm, we have
\[
n^{-1}\|\widetilde{X}^{(n)}\|_r^r
=\frac{1}{n}\sum_{i=1}^{n}\Big((\widetilde{X}_i^{(n)})^*\Big)^r.
\]
Hence, by Lemma~\ref{lem:uniform-app} together with the same arguments as in the proof of Theorem~\ref{thm:main},
\[
\Big| n^{-1}\|\widetilde{X}^{(n)}\|_r^r - \frac{1}{n}\sum_{i=1}^{n}G_q\Big(\frac{i-1}{n}\Big)^r\Big| \xrightarrow[n\to\infty]{\mathbb P} 0.
\]
It is therefore sufficient to compute the deterministic limit
\begin{equation} \label{eq:moments}
\lim_{n\to\infty}\frac{1}{n}\sum_{i=1}^{n}G_q\Big(\frac{i-1}{n}\Big)^r
= \int_0^1 G_q(x)^r \,\dd x
=
\begin{cases}
	\int_0^{1} \Big(\frac{1}{q-1}(1-x^{1-1/q})\Big)^{r}\,\dd x & \colon  q<\infty,\\
	\int_0^1 (1-x)^r \, \dd x & \colon  q=\infty.
\end{cases}
\end{equation}
The integral can be evaluated using a substitution and the beta function. 

The case $r=\infty$ follows directly from the fact that $\|\widetilde{X}^{(n)}\|_{\infty}=(\widetilde{X}^{(n)}_1)^*$ and Lemma~\ref{lem:uniform-app}.

By Corollary~\ref{cor:volume} it remains to multiply the resulting constants by $\frac{q}{2e^{1/q}}$ for $q<\infty$ and $\frac{1}{2}$ for $q=\infty$ in order to obtain the result for $X^{(n)}$ uniformly distributed in $\ID_{q,1}$.
\end{proof}

Finally, we present the proof of the threshold result for the asymptotic volume of intersections of normalized Lorentz and $\ell_r^n$ balls. The proof is based on the weak law of large numbers for the $\ell_r^n$ norm (see Theorem \ref{pro:lln-norm}). Recall that for each $r\in (0,\infty)$
\begin{equation}\label{eq:dirichlet}
	\vol_n(\IB_r^n)^{1/n} = \frac{2\Gamma(1+\frac{1}{r})}{\Gamma\big(1+\frac{n}{r}\big)^{1/n}}\sim 2\Gamma\big(1+\frac{1}{r}\big)(er)^{1/r}n^{-1/r}, \quad\text{as } n\to\infty,
\end{equation}
which is known at least since the work \cite{D1839} of Dirichlet.

\begin{proof}[Proof of Corollary~\ref{pro:intersect}]
We first prove the statement for $q<\infty$. For $X^{(n)}$ uniformly distributed on $\ID_{q,1}^n$, we can write
\[
\vol_n\big(\ID_{q,1}^n\cap t\ID_r^n\big)
=\IP\Bigg[\|X^{(n)}\|_r\le \frac{t}{\vol_n(\IB_r^n)^{1/n}}\Bigg].
\]
Then, simply rewriting the previous expression, we obtain
\[
\vol_n\big(\ID_{q,1}^n\cap t\ID_r^n\big) = \IP\Bigg[n^{-1/r}\|X^{(n)}\|_r-m_{q,r}\le \frac{t}{n^{1/r}\vol_n(\IB_r^n)^{1/n}}-m_{q,r}\Bigg],
\]
with $m_{q,r}$ as in Theorem \ref{pro:lln-norm}, in particular for $r<\infty$,
\[
m_{q,r}=\frac{1}{2e^{1/q}}\frac{q}{q-1}\left(\frac{\Gamma(r+1)\Gamma\Big(1+\frac{q}{q-1}\Big)}{\Gamma\Big(r+1+\frac{q}{q-1}\Big)}\right)^{1/r}.
\]
From the stated volume radius asymptotics in \eqref{eq:dirichlet}, we know that
\[
\frac{1}{n^{1/r}\vol_n(\IB_r^n)^{1/n}}
\stackrel{n\to\infty}{\longrightarrow} c_{q,r}:=
\begin{cases}
\frac{1}{2(er)^{1/r}\Gamma(1+1/r)} & \colon  r<\infty\\
\frac{1}{2}& \colon  r=\infty.
\end{cases}
\]
Therefore, in view of the weak law of large numbers in Theorem \ref{pro:lln-norm}, the probability tends to one if $tc_{q,r}>m_{q,r}$ and to zero if $tc_{q,r}<m_{q,r}$. Hence, the statement holds with threshold $A_{q,r}=c_{q,r}(m_{q,r})^{-1}$.

If $q=\infty$, we can proceed analogously using $m_{\infty,r}=\frac{1}{2}\Big(\frac{1}{r+1}\Big)^{1/r}$. This completes the proof.
%
%
\end{proof}

\section{The conjecture for general $p>1$ --- Maximum entropy heuristics}

Let us describe the heuristics leading us to Conjecture \ref{conj:allp} and specifically to the differential equation that appears in it. Our arguments are based on maximum entropy considerations \cite{RAS2015}, but are not mathematically rigorous (see \cite{KP2021} for a similar rigorous result in the case of Orlicz balls). We shall denote by $\mathscr{M}_1(\IR_+)$ the set of probability measures on $\IR_+$ equipped with the weak topology, which is a Polish space. Since Lorentz balls belong to the class of $1$-symmetric convex bodies, we can restrict considerations to the positive orthant (see also Lemma~\ref{lem:symmetry}) and consider a uniformly distributed random vector $X^{(n)}$ in
\[
\lop
=\Bigg\{x\in\IR^n_+\colon \frac{1}{n}\sum_{i=1}^{n}\Big(\frac{i}{n}\Big)^{p/q-1}|x_i^*|^p\le 1\Bigg\}.
\]
We now consider a sequence of independent and identically distributed ``random variables'' $Y_1,Y_2,\dots$ which we assume to be uniformly ``distributed'' according to the infinite Lebesgue measure $\lambda$ on $\IR_+$. Conditioned on $Y^{(n)}:=(Y_1,\dots,Y_n)\in \lop$, which is a so-called energy constraint, the ``random vector'' $Y^{(n)}$ is uniformly distributed in $\lop$. The maximum entropy principle states (under suitable conditions) that the random empirical probability measure $L_{Y^{(n)}}$ associated to $Y^{(n)}$, where for $x\in\IR^n$
\[
L_{x}:=\frac{1}{n}\sum_{i=1}^{n}\delta_{x_i},
\]
converges, when conditioned upon the rare event of remaining in a closed convex set $K\subset \mathscr{M}_1(\IR_+)$, to the measure $\mu^*$ minimizing the relative entropy
\[
H(\mu|\lambda)
:= \begin{cases}
	\int_0^{\infty} p(x)\log p(x) \,\lambda(\dd x) & : p=\frac{\dd \mu}{\dd \lambda} \text{ exists},\\
	\infty & \text{otherwise}
\end{cases}
\]
over $K$ (c.f. Sanov's theorem \cite[Section 5.2]{RAS2015} and its version for infinite measures \cite{BS2016}).
In order to rewrite the conditioning $Y^{(n)}\in \lop$ in terms of the empirical measure, we use the quantile function (i.e., the 
inverse cumulative distribution function), which for any probability measure $\mu\in\pr$ is given by
\[
Q_{\mu}\colon [0,1]\to [0,\infty],\quad Q_{\mu}(t):=\inf\big\{s\geq0\colon \mu( [0,s])\ge t\big\}.
\]
The following lemma follows from a simple rewriting which also appears in the analysis of L-statistics, see, e.g., \cite{Boi07}.

\begin{lemma} \label{lem:transfer}
Let $n\in\IN$. For any $x\in\IR_+^n$, we have
\[
x\in\lop \,\,\Longleftrightarrow\,\, L_{x}\in K_n:=\Bigg\{\mu\in\pr\colon \int_0^1 Q_{\mu}(t)^p J_n(t)\,\dd t\le 1\Bigg\},
\]
where
\[
J_n(t)
:=\sum_{i=1}^n\Big(1-\frac{i-1}{n}\Big)^{p/q-1} \mathbbm{1}_{(\frac{i-1}{n},\frac{i}{n}]}(t).
\]
\end{lemma}
\begin{proof}
By definition
\[
x\in\lop\quad
\Longleftrightarrow\quad \frac{1}{n}\sum_{i=1}^{n}\Big(\frac{i}{n}\Big)^{p/q-1}|x_i^*|^p\le 1.
\]
Moreover, we note that for $t\in[0,1]$
\[
Q_{L_{x}}(t)
=\inf\big\{s\geq 0 \,:\, L_{x}(	[0,s])\ge t\big\}
=\inf\Big\{s\geq 0\,\colon\, \frac{1}{n}\sum_{i=1}^{n}\delta_{x_i}([0,s])\ge t\Big\}\ge 0.
\]
The function $s\mapsto \frac{1}{n}\sum_{i=1}^{n}\delta_{x_i}([0,s])$ returns the proportion of coordinates which are at most $s$. So the quantile function evaluated at $t$ is the minimal $s$ such that this proportion still exceeds $t$. If $t\in (\frac{i-1}{n},\frac{i}{n}]$, then this must be the value $s=x_{n-i+1}^*$ (the $i^{\text{th}}$ smallest coordinate) since after this point at least $i$ of the $n$ coordinates are at most $s$ and this is not true for any smaller value. Therefore, we have
\[
Q_{L_{x}}(t)
=x_{n-i+1}^* \quad \text{for } t\in \Big(\frac{i-1}{n},\frac{i}{n}\Big],\,\,i\in \{1,\dots,n\}.
\]
Now we observe further that, for all $t\in [0,1]$,
\begin{align*}
\frac{1}{n}\sum_{i=1}^{n}\Big(\frac{i}{n}\Big)^{p/q-1}|x_i^*|^p
&= \frac{1}{n}\sum_{i=1}^{n}\Big(1-\frac{i-1}{n}\Big)^{p/q-1}|x_{n-i+1}^*|^p\\
&= \frac{1}{n}\sum_{i=1}^{n}J_n(t)Q_{L_{x}}(t)^p\mathbbm{1}_{(\frac{i-1}{n},\frac{i}{n}]}(t).
\end{align*}
Since both functions are constant on $t\in (\frac{i-1}{n},\frac{i}{n}]$, the factor $1/n$ can be interpreted as the integral over these functions on this interval, i.e., 
\[
\frac{1}{n}\sum_{i=1}^{n}J_n(t)Q_{L_{x}}(t)^p\mathbbm{1}_{(\frac{i-1}{n},\frac{i}{n}]}(t)
=\sum_{i=1}^{n}\int_{\frac{i-1}{n}}^{\frac{i}{n}}J_n(t)Q_{L_{x}}(t)^p\,\dd t
=\int_{0}^{1}J_n(t)Q_{L_{x}}(t)^p\,\dd t.
\]
This immediately proves the lemma.
\end{proof}

In view of Lemma \ref{lem:transfer}, we have in our setting
\[
Y^{(n)}\in \lop
\,\,\Longleftrightarrow\,\, L_{Y^{(n)}}\in K_n.
\]
To apply variants of the Gibbs conditioning principle as the one used, e.g., in \cite{FP2021} and \cite{KR18}, we have to account for the fact that the set $K_n$ depends on $n$. For large $n$ it will be approximately equal to the set
\[
K
:=\Bigg\{\mu\in\pr\colon \int_0^1 Q_{\mu}(t)^p J(t)\,\dd t\le 1\Bigg\},
\]
where
\[
J\colon [0,1]\to\IR,\quad J(t)=(1-t)^{p/q-1},\quad t\in [0,1],
\]
is the pointwise limit of the sequence $J_n$, $n\in\IN$. We ignore the technical details.

In order to minimize the relative entropy over all measures in $K$, we can obviously ignore non-absolutely continuous measures for which the relative entropy is $\infty$. 
Let $F_\mu$ be the distribution function of the minimizer and 
$F_{\mu}'=\frac{\dd \mu}{\dd \lambda}$ its probability density function as well as $F_{\mu}^{-1}=Q_{\mu}$ the quantile function. Using the substitution $x=F^{-1}_{\mu}(y)$ and the fact that $Q_{\mu}'(y) = (F_{\mu}^{-1})'(y) = 1/F_{\mu}'(F_{\mu}^{-1}(y))$,
we can write
\[
H(\mu|\lambda)
=\int_0^{\infty} F_{\mu}'(x)\log F_{\mu}'(x) \,\lambda(\dd x)
=\int_0^{1} \log F_{\mu}'(F_{\mu}^{-1}(y))\, \lambda(\dd y)
=-\int_0^1 \log Q_{\mu}'(y)\, \lambda(\dd y).
\]
This leads to the following variational problem
\begin{align*}
	\max &\quad \int_0^{1}\log Q'(x)\, \dd x\\
	\text{s.t.}&\quad Q\colon [0,1]\to [0,\infty]\\
	&\quad Q\text{ is nondecreasing and differentiable}\\
	&\quad \int_0^1 Q(x)^p (1-x)^{p/q-1} \,\dd x\le 1.
\end{align*}
Its solution will be the quantile function of the minimizer $\mu^*\in K$ of the relative entropy and thus of the supposed limiting distribution. For the sake of convenience, we shall set $\alpha:=p/q-1\in [-1,0]$. In the following, we heuristically compute a solution to this problem and derive a differential equation for the corresponding distribution function.

Let us compute the variation and set $P:=Q+\varepsilon g$ for a small $\varepsilon>0$ and a continuously differentiable function $g\colon [0,1]\to\IR$ such that $g(0)=0$: 
\[
\int_0^1 \log (Q+\varepsilon g)'\dd x
=\int_0^{1}\log Q'(x) \,\dd x+\int_0^1\log \Big(1+\varepsilon \frac{g'(x)}{Q'(x)}\Big) \,\dd x.
\]
We expect that as $\varepsilon$ becomes small, the following approximation
\[
\int_0^1\log \Big(1+\varepsilon \frac{g'(x)}{Q'(x)}\Big)\,\dd x
= \varepsilon\int_0^1  \frac{g'(x)}{Q'(x)}\,\dd x +o(\varepsilon)
\]
is valid and hence
\[
\varepsilon^{-1}\Big(\int_0^1 \log (Q+\varepsilon g)'\,\dd x-\int_0^{1}\log Q'(x)\,\dd x\Big)
\stackrel{\varepsilon\to 0}{\longrightarrow}\int_0^1  \frac{g'(x)}{Q'(x)}\dd x.
\]
We now look at the constraints and compute
\[
\int_0^1 \big(Q(x)+\varepsilon g(x)\big)^p (1-x)^{\alpha} \,\dd x
=\int_0^1 Q(x)^p\Big(1+\varepsilon \frac{g(x)}{Q(x)}\Big)^p (1-x)^{\alpha} \,\dd x.
\]
Again, we expect that, as $\varepsilon$ becomes small, the approximation
\begin{align*}
\int_0^1 Q(x)^p\Big(1+\varepsilon \frac{g(x)}{Q(x)}\Big)^p (1-x)^{\alpha} \,\dd x
&= \int_0^1 Q(x)^p\Big(1+\varepsilon \frac{pg(x)}{Q(x)}\Big) (1-x)^{\alpha} \,\dd x + o(\varepsilon)\\
&=\int_0^1 Q(x)^p(1-x)^{\alpha}\dd x+\varepsilon \int_0^1 pg(x)Q(x)^{p-1} (1-x)^{\alpha} \,\dd x + o(\varepsilon)
\end{align*}
is valid. Hence, 
\[
\varepsilon^{-1}\Big(\int_0^1 (Q+\varepsilon g)^p (1-x)^{\alpha}\dd x-\int_0^{1} Q(x)^p (1-x)^{\alpha} \,\dd x\Big)
\stackrel{\varepsilon\to0}{\longrightarrow}\int_0^1 pg(x)Q(x)^{p-1} (1-x)^{\alpha} \,\dd x.
\]

Therefore, we only allow functions $g:[0,1]\to\IR$ with $g(0)=0$ and
\[
p\int_0^1 g(x)Q(x)^{p-1} (1-x)^{\alpha}\,\dd x=0.
\]
A necessary condition for the maximizer is that for every such $g$,
\[
0=\int_0^1  \frac{g'(x)}{Q'(x)} \,\dd x
=\frac{g(1)}{Q'(1)}-\int_0^1 g(x)\Big(\frac{1}{Q'}\Big)'(x) \,\dd x = \int_0^1 g(x)\Big(\frac{\delta_{1}(x)}{Q'(1)}-\Big(\frac{1}{Q'}\Big)'(x) \Big)\,\dd x,
\]
where in the second equality we used partial integration. We now use the following result.
\begin{lemma} \label{lem:variation}
Let $\mu_1,\mu_2$ be signed measures on $[0,1]$ such that if $\int_0^1 g(x)\mu_1(\dd x) =0 $ for some $g\in C[0,1]$ with $g(0)=0$, then also $\int_0^1 g(x)\mu_2(\dd x) =0$. Then, $\mu_2 = \alpha \mu_1+\beta \delta_0$ for some constants $\alpha,\beta \in\IR$.
\end{lemma}

Applied to our situation above, 
%
%
we conclude that there exists some $c\neq 0$ such that, for all $x\in [0,1]$,
\[
\Big(\frac{1}{Q'}\Big)'(x)=cQ(x)^{p-1} (1-x)^{\alpha}\quad\text{and}\quad\frac{1}{Q'(1)}=0.
\]
A substitution yields
\[
F''(x)=cF'(x)(1-F(x))^{\alpha}x^{p-1}, \quad x\in \IR_+.
\]
We must have $F(0)=0$ and there exists $r>0$ with $F(x)=1$ for $x\ge r$ ($r=+\infty$ is possible). Also the condition on $Q$ translates to (if we maximize)
\[
1= \int_0^r x^p (1-F(x))^{\alpha}F'(x) \,\dd x = \frac{1}{c}\int_0^r x F''(x)\,\dd x.
\]
Using that $F'(r)=0$, which follows from the fact $Q'(1)= + \infty$, together with integration by parts, we see that the right-hand side is $-\frac{1}{c}$. Hence, we conclude $c=-1$ leading to the desired differential equation. 

\bibliographystyle{plain}
\bibliography{lorentz}

\begin{thebibliography}{10}

\bibitem{AS64}
M.~Abramowitz and I.~A. Stegun.
\newblock {\em Handbook of mathematical functions with formulas, graphs, and
  mathematical tables}.
\newblock National Bureau of Standards Applied Mathematics Series, No. 55. U.
  S. Government Printing Office, Washington, D.C., 1964.

\bibitem{AP2022}
D.~Alonso-Guti\'{e}rrez and J.~Prochno.
\newblock Thin-shell concentration for random vectors in {O}rlicz balls via
  moderate deviations and {G}ibbs measures.
\newblock {\em J. Funct. Anal.}, 282(1):Paper No. 109291, 35, 2022.

\bibitem{APT2019}
D.~Alonso-Guti\'{e}rrez, J.~Prochno, and C.~Th\"{a}le.
\newblock Gaussian fluctuations for high-dimensional random projections of
  {$\ell_p^n$}-balls.
\newblock {\em Bernoulli}, 25(4A):3139--3174, 2019.

\bibitem{APT2021}
D.~Alonso-Guti\'{e}rrez, J.~Prochno, and C.~Th\"{a}le.
\newblock Large deviations, moderate deviations, and the {KLS} conjecture.
\newblock {\em J. Funct. Anal.}, 280(1):108779, 33, 2021.

\bibitem{A1975}
Z.~Altshuler.
\newblock Uniform convexity in {L}orentz sequence spaces.
\newblock {\em Israel J. Math.}, 20(3-4):260--274, 1975.

\bibitem{BS2016}
V.~Bakhtin and E.~Sokal.
\newblock The {K}ullback–{L}eibler information function for infinite
  measures.
\newblock {\em Entropy}, 18(12):448, 2016.

\bibitem{BGMN2005}
F.~Barthe, O.~Gu{\'e}don, S.~Mendelson, and A.~Naor.
\newblock A probabilistic approach to the geometry of the {$\ell^n_p$}-ball.
\newblock {\em Ann. Probab.}, 33(2):480--513, 2005.

\bibitem{BW2021}
F.~{Barthe} and P.~{Wolff}.
\newblock {Volume properties of high-dimensional Orlicz balls}.
\newblock {\em arXiv e-prints}, page arXiv:2106.01675, June 2021.

\bibitem{BS1988}
C.~Bennett and R.~Sharpley.
\newblock {\em Interpolation of operators}, volume 129 of {\em Pure and Applied
  Mathematics}.
\newblock Academic Press, Inc., Boston, MA, 1988.

\bibitem{BGT87}
N.~H. Bingham, C.~M. Goldie, and J.~L. Teugels.
\newblock {\em Regular variation}, volume~27 of {\em Encyclopedia of
  Mathematics and its Applications}.
\newblock Cambridge University Press, Cambridge, 1987.

\bibitem{Boi07}
H.~Boistard.
\newblock Large deviations for {$L$}-statistics.
\newblock {\em Statist. Decisions}, 25(2):89--125, 2007.

\bibitem{BGVV2014}
S.~Brazitikos, A.~Giannopoulos, P.~Valettas, and B.-H. Vritsiou.
\newblock {\em Geometry of {I}sotropic {C}onvex {B}odies}, volume 196 of {\em
  Mathematical Surveys and Monographs}.
\newblock American Mathematical Society, Providence, RI, 2014.

\bibitem{Das08}
A.~DasGupta.
\newblock {\em Asymptotic theory of statistics and probability}.
\newblock Springer Texts in Statistics. Springer, New York, 2008.

\bibitem{DN03}
H.~A. David and H.~N. Nagaraja.
\newblock {\em Order statistics}.
\newblock Wiley Series in Probability and Statistics. Wiley-Interscience [John
  Wiley \& Sons], Hoboken, NJ, third edition, 2003.

\bibitem{DZ10}
A.~Dembo and O.~Zeitouni.
\newblock {\em Large deviations techniques and applications}, volume~38 of {\em
  Stochastic Modelling and Applied Probability}.
\newblock Springer-Verlag, Berlin, 2010.
\newblock Corrected reprint of the second (1998) edition.

\bibitem{dVL1993}
R.~A. DeVore and G.~G. Lorentz.
\newblock {\em Constructive approximation}, volume 303 of {\em Grundlehren der
  mathematischen Wissenschaften [Fundamental Principles of Mathematical
  Sciences]}.
\newblock Springer-Verlag, Berlin, 1993.

\bibitem{D1839}
P.~G.~L. Dirichlet.
\newblock Sur une nouvelle méthode pour la détermination des intégrales
  multiples.
\newblock {\em Journal de Mathématiques Pures et Appliquées}, pages 164--168,
  1839.

\bibitem{DV2020}
A.~Dole\v{z}alov\'{a} and J.~Vyb\'{\i}ral.
\newblock On the volume of unit balls of finite-dimensional {L}orentz spaces.
\newblock {\em J. Approx. Theory}, 255:105407, 20, 2020.

\bibitem{EN11}
D.~E. Edmunds and Y.~Netrusov.
\newblock Entropy numbers and interpolation.
\newblock {\em Math. Ann.}, 351(4):963--977, 2011.

\bibitem{FR2013}
S.~Foucart and H.~Rauhut.
\newblock {\em A mathematical introduction to compressive sensing}.
\newblock Applied and Numerical Harmonic Analysis. Birkh\"{a}user/Springer, New
  York, 2013.

\bibitem{FP2021}
L.~{Fr\"uhwirth} and J.~{Prochno}.
\newblock {Sanov-type large deviations and conditional limit theorems for
  high-dimensional Orlicz balls}.
\newblock {\em arXiv e-prints}, page arXiv:2111.04691, November 2021.

\bibitem{GKR2017}
N.~Gantert, S.~S. Kim, and K.~Ramanan.
\newblock Large deviations for random projections of {$\ell^p$} balls.
\newblock {\em Ann. Probab.}, 45(6B):4419--4476, 2017.

\bibitem{G2014}
L.~Grafakos.
\newblock {\em Modern {F}ourier analysis}, volume 250 of {\em Graduate Texts in
  Mathematics}.
\newblock Springer, New York, third edition, 2014.

\bibitem{HKTJ2006}
C.~Hao, A.~Kami\'{n}ska, and N.~Tomczak-Jaegermann.
\newblock Orlicz spaces with convexity or concavity constant one.
\newblock {\em J. Math. Anal. Appl.}, 320(1):303--321, 2006.

\bibitem{HS1975}
E.~Hewitt and K.~Stromberg.
\newblock {\em Real and abstract analysis}.
\newblock Graduate Texts in Mathematics, No. 25. Springer-Verlag, New
  York-Heidelberg, 1975.
\newblock A modern treatment of the theory of functions of a real variable,
  Third printing.

\bibitem{HM2006}
A.~Hinrichs and C.~Michels.
\newblock Gelfand numbers of identity operators between symmetric sequence
  spaces.
\newblock {\em Positivity}, 10(1):111--133, 2006.

\bibitem{JMST1979}
W.~B. Johnson, B.~Maurey, G.~Schechtman, and L.~Tzafriri.
\newblock Symmetric structures in {B}anach spaces.
\newblock {\em Mem. Amer. Math. Soc.}, 19(217):v+298, 1979.

\bibitem{JP2023}
S.~G.~G. Johnston and J.~Prochno.
\newblock A {M}axwell principle for generalized {O}rlicz balls.
\newblock {\em Ann. Inst. H. Poincar\'{e} Probab. Statist. (to appear)}, pages
  1--25, 2023+.

\bibitem{JKP2022}
M.~{Juhos}, Z.~{Kabluchko}, and J.~{Prochno}.
\newblock {Limit theorems for mixed-norm sequence spaces with applications to
  volume distribution}.
\newblock {\em arXiv e-prints}, page arXiv:2209.08937, September 2022.

\bibitem{JP2022}
M.~Juhos and J.~Prochno.
\newblock Spectral flatness and the volume of intersections of
  {{\(p\)}}-ellipsoids.
\newblock {\em J. Complexity}, 70:21, 2022.
\newblock Id/No 101617.

\bibitem{KP2021}
Z.~Kabluchko and J.~Prochno.
\newblock The maximum entropy principle and volumetric properties of {O}rlicz
  balls.
\newblock {\em J. Math. Anal. Appl.}, 495(1):Paper No. 124687, 19, 2021.

\bibitem{KPT2019_I}
Z.~Kabluchko, J.~Prochno, and C.~Th\"{a}le.
\newblock High-dimensional limit theorems for random vectors in
  {$\ell_p^n$}-balls.
\newblock {\em Commun. Contemp. Math.}, 21(1):1750092, 30, 2019.

\bibitem{KPT2020_matrix}
Z.~Kabluchko, J.~Prochno, and C.~Th{\"a}le.
\newblock Intersection of unit balls in classical matrix ensembles.
\newblock {\em Isr. J. Math.}, 239(1):129--172, 2020.

\bibitem{K1984}
A.~Kami\'{n}ska.
\newblock The criteria for local uniform rotundity of {O}rlicz spaces.
\newblock {\em Studia Math.}, 79(3):201--215, 1984.

\bibitem{KLR2022}
S.~S. Kim, Y.-T. Liao, and K.~Ramanan.
\newblock An asymptotic thin shell condition and large deviations for random
  multidimensional projections.
\newblock {\em Adv. in Appl. Math.}, 134:Paper No. 102306, 64, 2022.

\bibitem{KR18}
S.~S. Kim and K.~Ramanan.
\newblock A conditional limit theorem for high-dimensional {$\ell^p$}-spheres.
\newblock {\em J. Appl. Probab.}, 55(4):1060--1077, 2018.

\bibitem{Kle08}
A.~Klenke.
\newblock {\em Probability theory}.
\newblock Universitext. Springer-Verlag London, Ltd., London, 2008.
\newblock A comprehensive course, Translated from the 2006 German original.

\bibitem{KMW2011}
P.~Kosmol and D.~M\"uller-Wichards.
\newblock {\em Optimization in Function Spaces}.
\newblock De Gruyter, Berlin, New York, 2011.

\bibitem{KS1985}
S.~Kwapie\'{n} and C.~Sch\"{u}tt.
\newblock Some combinatorial and probabilistic inequalities and their
  application to {B}anach space theory.
\newblock {\em Studia Math.}, 82(1):91--106, 1985.

\bibitem{LT1977}
J.~Lindenstrauss and L.~Tzafriri.
\newblock {\em Classical {B}anach spaces. {I}}.
\newblock Ergebnisse der Mathematik und ihrer Grenzgebiete, Band 92.
  Springer-Verlag, Berlin-New York, 1977.
\newblock Sequence spaces.

\bibitem{L1950}
G.~G. Lorentz.
\newblock Some new functional spaces.
\newblock {\em Ann. Math.}, 51(1):37--55, 1950.

\bibitem{L1951}
G.~G. Lorentz.
\newblock {On the theory of spaces $\Lambda$.}
\newblock {\em Pacific J. Math.}, 1(3):411 -- 429, 1951.

\bibitem{M1939}
J.~Marcinkiewicz.
\newblock {Sur l'interpolation d'opérations}.
\newblock {\em C.R. Acad. Sci. Paris}, 208:1272--1273, 1939.

\bibitem{MathaiBook}
A.~M. Mathai.
\newblock {\em An {I}ntroduction to {G}eometrical {P}robability}, volume~1 of
  {\em Statistical Distributions and Models with Applications}.
\newblock Gordon and Breach Science Publishers, Amsterdam, 1999.
\newblock Distributional aspects with applications.

\bibitem{MS01}
M.~M. Meerschaert and H.-P. Scheffler.
\newblock {\em Limit distributions for sums of independent random vectors}.
\newblock Wiley Series in Probability and Statistics: Probability and
  Statistics. John Wiley \& Sons, Inc., New York, 2001.
\newblock Heavy tails in theory and practice.

\bibitem{N2007}
A.~Naor.
\newblock The surface measure and cone measure on the sphere of {$l_p^n$}.
\newblock {\em Trans. Amer. Math. Soc.}, 359(3):1045--1079, 2007.

\bibitem{NR2003}
A.~Naor and D.~Romik.
\newblock Projecting the surface measure of the sphere of {$ \ell_p^n$}.
\newblock {\em Ann. Inst. H. Poincar\'{e} Probab. Statist.}, 39(2):241--261,
  2003.

\bibitem{P2020}
J.~Prochno.
\newblock Embeddings of {O}rlicz-{L}orentz spaces into {$L_1$}.
\newblock {\em Algebra i Analiz}, 32(1):78--93, 2020.

\bibitem{PS2012}
J.~Prochno and C.~Sch\"{u}tt.
\newblock Combinatorial inequalities and subspaces of {$L_1$}.
\newblock {\em Studia Math.}, 211(1):21--39, 2012.

\bibitem{PTT2020}
J.~{Prochno}, C.~{Th{\"a}le}, and N.~{Turchi}.
\newblock Geometry of $\ell_p^n$-balls:\,\,classical results and recent
  developments.
\newblock In {\em Progress in Probability}, High Dimensional Probability VIII.
  Birkh{\"a}user, 2020.

\bibitem{RR1991}
S.T. Rachev and L.~R{\"u}schendorf.
\newblock Approximate independence of distributions on spheres and their
  stability properties.
\newblock {\em Ann. Probab.}, 19(3):1311--1337, 1991.

\bibitem{RAS2015}
F.~Rassoul-Agha and T.~Sepp\"{a}l\"{a}inen.
\newblock {\em A course on large deviations with an introduction to {G}ibbs
  measures}, volume 162 of {\em Graduate Studies in Mathematics}.
\newblock American Mathematical Society, Providence, RI, 2015.

\bibitem{R1981}
S.~Reisner.
\newblock A factorization theorem in {B}anach lattices and its application to
  {L}orentz spaces.
\newblock {\em Ann. Inst. Fourier (Grenoble)}, 31(1):viii, 239--255, 1981.

\bibitem{R1982}
S.~Reisner.
\newblock On the duals of {L}orentz function and sequence spaces.
\newblock {\em Indiana Univ. Math. J.}, 31(1):65--72, 1982.

\bibitem{SS1991}
G.~Schechtman and M.~Schmuckenschl\"ager.
\newblock Another remark on the volume of the intersection of two {$L^n_p$}
  balls.
\newblock In {\em Geometric aspects of functional analysis (1989--90)}, volume
  1469 of {\em Lecture Notes in Math.}, pages 174--178. Springer, Berlin, 1991.

\bibitem{SZ1990}
G.~Schechtman and J.~Zinn.
\newblock On the volume of the intersection of two {$L^n_p$} balls.
\newblock {\em Proc. Amer. Math. Soc.}, 110(1):217--224, 1990.

\bibitem{SZ2000}
G.~Schechtman and J.~Zinn.
\newblock Concentration on the {$l^n_p$} ball.
\newblock In {\em Geometric aspects of functional analysis}, volume 1745 of
  {\em Lecture Notes in Math.}, pages 245--256. Springer, Berlin, 2000.

\bibitem{S2001}
M.~Schmuckenschl{\"a}ger.
\newblock {CLT} and the volume of intersections of {{\(\ell_p^n\)}}-balls.
\newblock {\em Geom. Dedicata}, 85(1-3):189--195, 2001.

\bibitem{Sch1982}
C.~Sch\"utt.
\newblock On the volume of unit balls in {B}anach spaces.
\newblock {\em Compositio Math.}, 47(3):393--407, 1982.

\bibitem{Sch1984}
C.~Sch\"{u}tt.
\newblock Entropy numbers of diagonal operators between symmetric {B}anach
  spaces.
\newblock {\em J. Approx. Theory}, 40(2):121--128, 1984.

\bibitem{Sch1989}
C.~Sch\"{u}tt.
\newblock Lorentz spaces that are isomorphic to subspaces of {$L^1$}.
\newblock {\em Trans. Amer. Math. Soc.}, 314(2):583--595, 1989.

\bibitem{Sch1995}
C.~Sch\"{u}tt.
\newblock On the embedding of {$2$}-concave {O}rlicz spaces into {$L^1$}.
\newblock {\em Studia Math.}, 113(1):73--80, 1995.

\bibitem{Szn91}
A.-S. Sznitman.
\newblock Topics in propagation of chaos.
\newblock In {\em \'{E}cole d'\'{E}t\'{e} de {P}robabilit\'{e}s de
  {S}aint-{F}lour {XIX}---1989}, volume 1464 of {\em Lecture Notes in Math.},
  pages 165--251. Springer, Berlin, 1991.

\bibitem{TJ1989}
N.~Tomczak-Jaegermann.
\newblock {\em Banach-{M}azur distances and finite-dimensional operator
  ideals}, volume~38 of {\em Pitman Monographs and Surveys in Pure and Applied
  Mathematics}.
\newblock Longman Scientific \& Technical, Harlow; copublished in the United
  States with John Wiley \& Sons, Inc., New York, 1989.

\end{thebibliography}

\bigskip
\bigskip
	
	\medskip
	
	\small
	\noindent \textsc{Zakhar Kabluchko:} Faculty of Mathematics and Computer Science,
		University of M\"unster, Orl\'eans-Ring 10,
		48149 M\"unster, Germany.
		
		\noindent
		{\it E-mail:} \texttt{zakhar.kabluchko@uni-muenster.de}
	
	\bigskip
	\noindent \textsc{Joscha Prochno:} Faculty of Computer Science and Mathematics,
	University of Passau, Dr.-Hans-Kapfinger-Stra{\ss}e 30, 94032 Passau, Germany.
	
	\noindent
	{\it E-mail:} \texttt{joscha.prochno@uni-passau.de}
	
	\bigskip
	\noindent \textsc{Mathias Sonnleitner:} Faculty of Computer Science and Mathematics,
		University of Passau, Dr.-Hans-Kapfinger-Stra{\ss}e 30, 94032 Passau, Germany.
		
		\noindent
		{\it E-mail:} \texttt{mathias.sonnleitner@uni-passau.de}

\end{document}